\documentclass[a4paper,11pt,reqno]{amsart}

\usepackage[latin1]{inputenc}
\usepackage{subfigure}
\usepackage{amsmath, amssymb, mathtools}
\usepackage{algorithm}
\usepackage{algpseudocode}
\usepackage{comment}
\usepackage{float}
\usepackage{tikz}
\usepackage{url}
\usepackage{fullpage}
\usepackage{bbm}
\usepackage{graphicx}

\newtheorem{Thm}{Theorem}[section]
\newtheorem{Cor}[Thm]{Corollary}
\newtheorem{Lem}[Thm]{Lemma}
\newtheorem{Prop}[Thm]{Proposition}

\newtheorem{Def}[Thm]{Definition}

\newtheorem{Rem}[Thm]{Remark}

\definecolor{dgreen}{rgb}{0,0.8,0}
\definecolor{ddgreen}{rgb}{0,0.5,0}

\newcommand{\Frame}[3]{\mathsf{Frame}_{#1 , #2}^{#3}}
\newtheorem{Ex}[Thm]{Example}

\def\Stable{\mathsf{Stable}}

\def\lesquare{\unlhd}
\def\gesquare{\unrhd}
\def\nlesquare{\not\!\!\lesquare\,\,\,}
\def\ngesquare{\not\!\gesquare\,\,}

\newcommand{\pos}{\mathrm{pos}}
\newcommand{\grade}{\mathrm{grade}}
\newcommand{\parking}{\mathrm{parking}}
\newcommand{\recurrent}{\mathrm{recurrent}}
\newcommand{\cumuledpos}{\mathrm{cumuledpos}}
\newcommand{\mypath}[2]{[#2]_{#1}}
\newcommand{\mypathfactor}[3]{[#2]_{{#1}|{#3}}}

\newcommand{\myframeconfiguration}[3]{\mathrm{Gauge}_{{#1},{#2}}(#3)}
\newcommand{\Para}{\mathsf{Polyo}}

\usepackage{color}

\newcommand{\groupGG}{\mathcal{G}}
\newcommand{\GG}{{G}}

\newcommand{\ZZ}{\mathbb{Z}}

\newcommand{\figureornotfigure}[1]{#1} 

\begin{document}

\author[J-C. Aval]{Jean-Christophe Aval}
\thanks{J.-C. Aval is supported by ANR - PSYCO project (ANR-11-JS02-001)}
\address{LaBRI, CNRS, Universit\'e de Bordeaux,
251 Cours de la Lib\'eration, Bordeaux, France.}

\author[M. D'Adderio]{Michele D'Adderio}
\address{Universit\'e Libre de Bruxelles (ULB), D\'epartement de Math\'ematique,
Boulevard du Triomphe, B-1050 Bruxelles, Belgium.}

\author[M.Dukes]{Mark Dukes}
\address{Department of Computer and Information Sciences,
University of Strathclyde, Glasgow G1 1XH, United Kingdom.}

\author[Y. Le Borgne]{Yvan Le Borgne}
\address{PIMS, CNRS, Simon Fraser University, Burnaby BC, Canada\newline \& LaBRI, CNRS, Universit\'e de Bordeaux,
251 Cours de la Lib\'eration, Bordeaux, France.}

\title[Two operators on sandpile configurations]{Two operators on sandpile configurations, the sandpile model on the complete bipartite graph, and a Cyclic Lemma}

\date{\today}


\maketitle

\begin{abstract}
  We introduce two operators on stable configurations of the sandpile
  model that provide an algorithmic bijection between recurrent and
  parking configurations. This bijection preserves their equivalence
  classes with respect to the sandpile group. The study of these operators in the
  special case of the complete bipartite graph ${K}_{m,n}$ naturally
  leads to a generalization of the well known Cyclic Lemma of
  Dvoretsky and Motzkin, via pairs of periodic bi-infinite paths in
  the plane having slightly different slopes. We achieve our results by 
  interpreting the action of these operators as an action on a point in the 
  grid $\mathbb{Z}^2$ which is pointed to by one of these pairs of paths.
  Our Cyclic lemma allows us to enumerate several classes of polyominoes, and 
  therefore builds on the work of Irving and Rattan (2009), Chapman et al. (2009), 
  and Bonin et al. (2003).
\end{abstract}

\section{Introduction}
\label{sec:in} The abelian sandpile model is a cellular automaton
on a graph. It was the first example of a dynamical system
exhibiting a fascinating property called {\it{self-organized
criticality}}, see \cite{SOC}. This model has since proved to be a
fertile ground from which many new and unlikely results have
emerged. One popular example is the correspondence between
recurrent configurations of the sandpile model on a graph and
spanning trees of the same graph, see e.g.~\cite{DHAR}.

In the abelian sandpile model on a (undirected connected
loop-free) graph, states are vectors which indicate the number of
grains present at every vertex of the graph. A vertex may be
toppled when the number of grains at that vertex is not less than
the degree of that vertex. When a vertex is toppled, one grain of
sand is sent along each incident edge to neighboring vertices. A
{\it{sink}} is a distinguished vertex in the graph. A
{\it{configuration}} is an assignment of grains to graph vertices, and 
a configuration is called {\it{stable}} if the number of grains at
each vertex other than the sink is less than the degree of that
vertex.

Two configurations are called {\it{toppling equivalent}} if there
is a sequence of topplings of one of the configurations that
results in the other. Given a configuration, the configurations
that can be obtained from it by any finite sequence of topplings
form the toppling equivalence class of this configuration. 
We study in particular the partition of stable configurations 
into toppling equivalence classes.

In Section 2 of this paper we consider two operators, $\psi$ and
$\varphi$, on stable sandpile configurations. These operators are,
in a sense, dual to one another. We prove that the fixed points of
the operator $\psi$ are the recurrent sandpile configurations and
the fixed points of $\varphi$ are the $\GG$-parking sandpile
configurations (an extension of the classical parking function to
an arbitrary directed graph $G$). 
The motivation in introducing the operators $\psi$
and $\varphi$ was to produce an algorithm that allows one to go
from recurrent configurations to $\GG$-parking configurations, and
vice-versa, within the same toppling equivalence class. As a
byproduct, we get two dual definitions of recurrent and
$\GG$-parking configurations.

In Section 3, we consider pairs of periodic bi-infinite paths in the plane defined by a pair of binary words. 
These binary words describe their respective (not necessarily) minimal periods.  
The two periods differ slightly since one period describes a lattice path from the origin to $(m,n)$ while the other describes a path from the origin to $(m-1,n)$.
For both of these finite paths, amend and prepend that same path to itself an infinite number of times to produce a pair of periodic paths with slightly different periods.
We prove a result that we call the Cyclic Lemma (Lemma 3.1) which consists of two parts.
The first part shows that the pairs of binary words can be partitioned into sets that each contain $m$ elements -- the sets consist of those pairs of binary words which define \emph{up to translation} the same pair of bi-infinite paths.
The second part shows that in every such set, there is precisely one pair of binary words that, when restricted to the rectangle of corners $(0,0)$ and $m\times n$, form a parallelogram polyomino. 
An immediate corollary of our Cyclic Lemma is an enumeration of
parallelogram polyominoes having an $m\times n$ bounding box. 

Chottin~\cite{CHO75} presented a result on words that is similar in spirit to our Cyclic Lemma. 
Our procedure is specifically designed to suit classes that are relevant in the context of the sandpile model. 
Moreover, our presentation is the perfect tool to deal with {\em labelled} parallelogram polyominoes, which were recently investigated in \cite{ABG}. 

The paper~\cite{YM} showed how configurations 
of the sandpile model on the complete bipartite graph $K_{m,n}$ 
that are both stable and sorted 
may be 
viewed as collections of cells in the plane.
By sorted we mean that configuration heights are weakly increasing in each part with respect to vertex indices. 
This correspondence was shown to have the property that a configuration is recurrent
if and only if the collection of corresponding cells in the plane
forms a parallelogram polyomino whose bounding box is an $m\times
n$ rectangle.
In Section 4 we restrict our attention to the complete bipartite
graph, and give an algorithm which computes $\varphi$ for stable
configurations on $K_{m,n}$.

In Section 5 we bring together the results of Sections 2, 3 and 4
while also building on the construction given in \cite{YM}. We
will represent sandpile configurations on $K_{m,n}$ as bi-infinite
pairs of paths in the plane. Configurations of the sandpile model
may be read from these bi-infinite paths by placing a `frame' at
certain points of intersection and performing measurements to
steps of the paths from this frame. We present and prove results
which show how the algorithmic calculation of $\varphi$ on
sandpile configurations in Section~4 may be interpreted as the
moving of the frame to a new point in the plane. We also give a
similar interpretation for $\psi$, and we deduce several
consequences of these results. Notably, we give a pictorial
description of $K_{m,n}$-parking configurations, in analogy with
the parallelogram polyominoes for recurrent configurations.

In Section 6 enumerative results about bi-infinite paths are given
in the case when one of the bi-infinite paths has a particularly regular
`staircase shape'.
Our work complements recent work of Irving and
Rattan~\cite{IR} and Chapman, Chow, Khetan, Moulton,
Waters~\cite{CCKMW} concerning the enumeration of lattice paths with
respect to a cyclically shifting boundary.

Finally, in Section 7 we show how results concerning the behavior of the
operators $\varphi$ and $\psi$ on the complete graph $K_{n}$ can
be derived from those on $K_{m,n}$.

The results in this paper arose from studying statistics on
parallelogram polyominoes and a symmetric functions interpretation
of a bi-statistic generating function that relates these statistics to
diagonal harmonics~\cite{ADDHL}. Throughout this paper, the phrase
`Cyclic Lemma' refers to our Lemma~\ref{theo:cyclic_lemma} unless
stated otherwise.

\section{Two operators on general sandpile configurations}

In this section we define two operators on sandpile
configurations for an undirected connected loop-free graph $G$. We
will show that the fixed points of these operators correspond to
recurrent configurations and $\GG$-parking configurations of the
sandpile model on $G$. Following this we will make some
observations concerning the injectivity of these operators, and
present a relation between them in terms of an operator called
$\beta$. These results are necessary in order to deal with the
specialization of $G$ to the graph $K_{m,n}$ from Section 4
onwards.

Consider an undirected connected loop-free graph $G$ with vertex
set $V=\{v_1,v_2,\dots,v_{n+1}\}$. We call the vertex $v_{n+1}$
the \textit{sink}. Let $d_i$ be the degree of the vertex $v_i$,
and $e(i,j)\in\{0,1\}$ the indicator function of an edge between $v_i$ and $v_j$. 
Let $\alpha_i\in \mathbb{Z}^{n+1}$ be a vector with $1$ in the $i$-th
position and $0$ elsewhere. Define the \textit{toppling} operators
$\Delta_i:=d_i\alpha_i-\sum_{j\neq i}e(i,j)\alpha_j$ for
$i=1,2,\dots,n+1$. Notice that $\sum_{j=1}^{n+1}\Delta_j =0$.

A configuration on $G$ is a vector $c=(c_1,\ldots,c_{n+1}) \in
\mathbb{Z}^{n+1}$. Most of the time we will consider
configurations modulo the height of the sink $v_{n+1}$. Thus two
configurations will be called equal or equivalent if their heights
at all vertices, other than the sink, are the same. The number of
grains at the sink is immaterial and we often record this number
as `$*$'.

Let  $c=(c_1,c_2,\dots,c_{n+1})$ be a configuration on $G$. We
call the sum $\sum_{i=1}^n c_i$ the \textit{height} of the
configuration. We say that $c$ is \textit{non-negative} if
$c_i\geq 0$ for all $1\leq i\leq n$ and that $c$ is
{\it{semi-stable}} if $c_i<d_i$ for all $1\leq i\leq n$. If a
configuration $c$ is both non-negative and semi-stable, then we
call it {\it{stable}}.

If $c$ is a non-negative configuration such that $c-\Delta_i$ is still 
non-negative, then we say the vertex $v_i$ is \textit{unstable}. By
convention we can always topple the sink. Stable configurations
are the ones for which there is no vertex that can be toppled
except the sink. Let $\Stable(G)$ be the set of all stable
configurations on $G$.

\begin{Def}
A configuration $c \in \Stable(G)$ is \textit{recurrent} if, after
toppling the sink, there is an order of the remaining vertices in
which we can topple every vertex of $G$ exactly once in that order
(thereby arriving back to the original configuration $c$).
\end{Def}

Given $A\subseteq \{1,2,\dots,n,n+1\}$, we define
$\Delta_A:=\sum_{j\in A}\Delta_j$. By convention
$\Delta_{\emptyset}$ is the zero vector.

\begin{Def}[see \cite{PS}]
A non-negative configuration $c \in \Stable(G)$ is a
\textit{$\GG$-parking configuration} if the configuration
$c-\Delta_A$ is not non-negative, for all non-empty $A\subseteq
\{1,2,\dots,n\}$.
\end{Def}

In order to define the operator $\psi$ on stable configurations of
a graph $G$ we need the following terminology concerning orders of
vertices and sets thereof. Fix a total order $<_1$ on the
vertices, for example $v_1<_1 v_2<_1 \cdots<_1 v_{n+1}$, which
corresponds to the order $1<2<\cdots<n+1$ on the indices. Next
define the order $\prec$ on the subsets of $\{1,2,\dots,n\}$: if
$A$ and $B$ are subsets $\{1,2,\dots,n\}$, then $A \prec B$ if (i)
$|A|<|B|$, or (ii) if $|A|=|B|$ and $A$ is smaller than $B$ in the
lexicographic order induced by the fixed order $<_1$ on the
vertices.

\begin{Def}[of $\psi$]
\label{def_psi} Given $c \in \Stable (G)$, let
\begin{eqnarray*}
\psi(c) &=& \left\{
\begin{array}{ll}
c & \begin{array}{l}\mbox{if $c+\Delta_A \not\in\Stable(G)$
for all non-empty $A\subseteq \{1,2,\dots,n\}$}\end{array}\\[0.5em]
c+\Delta_A & \begin{array}{l}\mbox{otherwise, where
$A\subseteq \{1,2,\dots,n\}$ is non-empty and minimal}\\
\mbox{(w.r.t. $\prec$) such that $c+\Delta_A
\in\Stable(G)$.}\end{array}
\end{array}
\right.
\end{eqnarray*}
\end{Def}

\begin{Thm}
\label{psithm} The fixed points of $\psi$ are exactly the
recurrent configurations of $G$.
\end{Thm}

\begin{proof}
Suppose that $c \in \Stable(G)$ is recurrent. This means there
exists some permutation $(a_1,a_2,\dots,a_n)$ of $(1,2,\dots,n)$
such that $c-\Delta_{n+1}-\sum_{j=1}^i\Delta_{a_j}$ is
non-negative for all $i=0,1,2,\dots,n$. Let
$X:=\{1,2,\dots,n,n+1\}$, and suppose by contradiction that there
exists $A\subseteq \{1,2,\dots,n\}$ such that
$c+\Delta_A=c-\Delta_{X\setminus A}$ is stable.

Let $i$ be minimal such that $a_i\in A$. For $Y\subseteq X$ and
$v$ a vertex, let $\mathrm{deg}_Yv$ denote the number of edges
from $v$ to a vertex in $Y$ as the other vertex. 

Then, since $a_j\notin A$ for $j<i$, and by the fact that we could
topple $v_{a_i}$ in the configuration
$c-\Delta_{n+1}-\sum_{j=1}^{i-1}\Delta_{a_j}$, we have
\begin{eqnarray*}
c_{a_i}+\mathrm{deg}_{X\setminus A}v_{a_i} & \geq &
c_{a_i}+e(a_i,n+1)+\sum_{j=1}^{i-1}e(a_i,a_j)\\
& \geq & \mathrm{deg}_{X}v_{a_i}.
\end{eqnarray*}

But by the stability of $c+\Delta_A$ we must also have
$$c_{a_i}+\mathrm{deg}_Xv_{a_i}-\mathrm{deg}_{A}v_{a_i}
    =c_{a_i}+\mathrm{deg}_{X\setminus A}v_{a_i}\leq
    \mathrm{deg}_Xv_{a_i}-1.$$

Putting these together, we get
$$\mathrm{deg}_Av_{a_i}\leq c_{a_i}\leq \mathrm{deg}_Av_{a_i}-1,$$
a contradiction. Thus no such (non-empty) $A$ exists, and
$c+\Delta_A \not\in\Stable(G)$ for all non-empty $A \subseteq
\{1,\ldots,n\}$. Therefore $\psi(c)=c$ from
Definition~\ref{def_psi}.

Suppose now that $c$ is a fixed point of $\psi$. This means that
$c+\Delta_A=c-\Delta_{X\setminus A}$ is not stable for all
non-empty $A\subseteq \{1,2,\dots,n\}$. Since
$c-\Delta_{n+1}=c+\sum_{i\neq n+1}\Delta_{i}$ is not stable, it
has a vertex, say $v_{a_1}\neq v_{n+1}$, that can be toppled. But
then $c-\Delta_{n+1}-\Delta_{a_1}$ is also not stable and it has a
vertex, $v_{a_2}$ say, different from $v_{a_1}$ and $v_{n+1}$,
that can be toppled.

Iterating this argument we get a sequence $(a_1,a_2,\dots,a_n)$ of
distinct indices, hence a permutation of $\{1,2,\dots,n\}$, such
that, after toppling the sink, we can topple the other vertices in
that order. In other words, $c$ is recurrent.
\end{proof}

Notice that the condition in the theorem is related to the
so-called \emph{allowed configurations} (cf. \cite{PS}).

We now define another operator, $\varphi$, which is a sort of
``dual'' operator to $\psi$. We use the same total order $<_1$ on
the vertices that we used for $\psi$.

\begin{Def}[of $\varphi$]
\label{def_phi} Given $c \in \Stable (G)$, let
\begin{eqnarray*}
\varphi(c) &=& \left\{ \begin{array}{ll} c &
\begin{array}{l}\mbox{if $c-\Delta_A \not\in\Stable(G)$ for all
non-empty
    $A\subseteq \{1,2,\dots,n\}$}\end{array}\\[0.5em]
c-\Delta_A & \begin{array}{l}\mbox{otherwise, and $A\subseteq
\{1,2,\dots,n\}$
    is non-empty and minimal }\\
\mbox{(w.r.t.$\prec$) such that $c-\Delta_A \in \Stable(G)$.}
\end{array}
\end{array} \right.
\end{eqnarray*}
\end{Def}

\begin{Thm}
\label{varphithm} The fixed points of $\varphi$ are exactly the
$\GG$-parking configurations.
\end{Thm}

\begin{proof}
It is straightforward to show that a $\GG$-parking configuration
is a fixed point of $\varphi$ since stable configurations are
non-negative.

Suppose the converse is not true. Let $c$ be a fixed point of
$\varphi$, and suppose that there exists a non-empty $A\subseteq
\{1,2,\dots,n\}$ such that $c-\Delta_A$ is still non-negative.
Since $c$ is a fixed point of $\varphi$, $c-\Delta_A$ is also
unstable (but non-negative). Therefore there must be an $i$ such
that we can topple $v_i$. There are two cases to consider.

{\it Case  $i\in A$}: In this case, using the notation of
Theorem~\ref{psithm}, instability implies
$$
c_i-\mathrm{deg}_{X}v_i+\mathrm{deg}_{A}v_i=c_i-\mathrm{deg}_{X\setminus
A}v_i\geq \mathrm{deg}_Xv_i.
$$
However $c_i-\mathrm{deg}_{X\setminus A}v_i\leq c_i\leq
\mathrm{deg}_Xv_i-1$, since $c$ is stable. This gives a
contradiction. Therefore $i$ cannot be in $A$.

{\it Case $i\notin A$}: In this case consider
$c-\Delta_{A\cup\{i\}}$. This vector (configuration) is
non-negative, since $c-\Delta_A$ was non-negative, and we could
topple $v_i$, but it is also unstable, since $c$ is a fixed point
of $\varphi$.

Iterating this argument, we can enlarge our set $A$ until we get
to the point were the second case does not occur. But this will
give a contradiction. This completes the proof.
\end{proof}

\begin{Rem}
The last two theorems provide a perfect duality between the
definitions of recurrent and $G$-parking configurations, a
desirable fact that partially motivated our investigations.
\end{Rem}

Set $\groupGG:=\mathbb{Z}^{n+1}/\langle \alpha_{n+1}\rangle$. We
will call $\groupGG/\langle \{\Delta_j\}_{j=1}^{n+1}\rangle$ the
\textit{sandpile group}. We call the cosets of the sandpile group
\textit{classes} so that we can talk about the class of a
configuration. It is well known that in each class there is
exactly one recurrent configuration and exactly one $\GG$-parking
configuration (see for example \cite[Theorem 1]{cori} and
\cite[Proposition 3.1]{baker}). There is an easy bijection between
recurrent and $\GG$-parking configurations (cf. \cite[Lemma
5.6]{baker}), but under this bijection configurations which
correspond to one-another do not necessarily lie in the same
class: see Remark~\ref{rem:beta}.

However, as we stated in the introduction, our motivation in
introducing the operators $\psi$ and $\varphi$ was to produce an
algorithm that allows one to pass from a recurrent configuration
to a $\GG$-parking configuration in the same class, and vice
versa.

For a configuration $c=(c_1,\ldots,c_{n+1})$ on an undirected
connected loop-free graph $G$, let $D(c)=(d_0,d_1,\ldots)$ be the
distribution of the distances of grains to the sink $v_{n+1}$. The
distance of $v_i$ to $v_{n+1}$ is the minimal number of edges on a
path from $v_i$ to $v_{n+1}$ in $E(G)$. In other words
$d_k=\sum_{v_i} c_i$ where $v_i$ runs over vertices whose distance
from the sink $v_{n+1}$ is $k$. Note that $d_0=c_{n+1}$.

\begin{Def}[of $<_2$]
Let $G$ be a graph with $V(G)=\{v_1,\ldots,v_{n+1}\}$. Let
$c=(c_1,\ldots,c_{n+1})$ and $c'=(c'_1,\ldots,c'_{n+1})$ be two
configurations in $\Stable(G)$ with $D(c)=(d_0,d_1,\ldots)$ and
$D(c')=(d'_0,d'_1,\ldots )$. If $D(c)$ is lexicographically
smaller then $D(c')$ then we write $c<_2 c'$.
\end{Def}

Observe that when applied to any configuration which is not recurrent,
the operator $\psi$ is strictly decreasing with respect to the order $<_2$.
Therefore, if we start with a $\GG$-parking
configuration, iterating the operator $\psi$ we will get a recurrent
configuration in finitely many steps. 
In the same way, when applied to a any configuration which is not $\GG$-parking,
the operator $\varphi$ is strictly increasing with respect to the order $<_2$.
Thus, starting from a recurrent configuration, and iterating the operator $\varphi$ we will
get a $\GG$-parking configuration in finitely many steps.
As a consequence, we have a bijection
between the recurrent configurations of $G$ and the $\GG$-parking
configurations that clearly preserves the classes.

\begin{Rem}\label{rem:rem}
The operator $\psi$ (resp.  $\varphi$) is, in general, not
injective even if restricted to the stable configurations that are
not recurrent (resp.  $\GG$-parking). However, we will see that
both in the case of ${K}_{n+1}$ and in the case of ${K}_{m,n}$, if
$c$ is a stable configuration which is not recurrent, then
$\varphi(\psi(c))=c$, and if $c$ is a stable configuration which
is not parking, then $\psi(\varphi(c))=c$.

Moreover, in these cases the operators $\psi$ and $\varphi$ are
inverses of each others in the sense of semigroups, i.e.
$\psi(\varphi(\psi(c)))=\psi(c)$ and
$\varphi(\psi(\varphi(c)))=\varphi(c)$ for all stable
configurations $c$. See Remark~\ref{rem:inv}.
\end{Rem}

\begin{Ex}\label{ex:graph}
\protect{\rm{ In this example we illustrate two applications of
Definition~\ref{def_psi} to sandpile configurations on a graph.
These examples will then be used to show that $\psi$ need not be
injective {\bf{even}} if restricted to non-recurrent
configurations. Consider the graph $G=(V,E)$ with
$V=\{v_1,\ldots,v_7\}$ and
$$E=\left\{ \{v_1,v_2\}, \{v_1,v_3\}, \{v_2,v_3\}, \{v_3,v_4\}, \{v_4,v_5\}, \{v_4,v_6\},
\{v_5,v_6\}, \{v_5,v_7\}, \{v_6,v_7\} \right\}.$$ Let vertex $v_7$
be the sink. 
This graph is illustrated by Figure \ref{fig:example-graph}.
\begin{figure}[H]
\includegraphics[width=6cm]{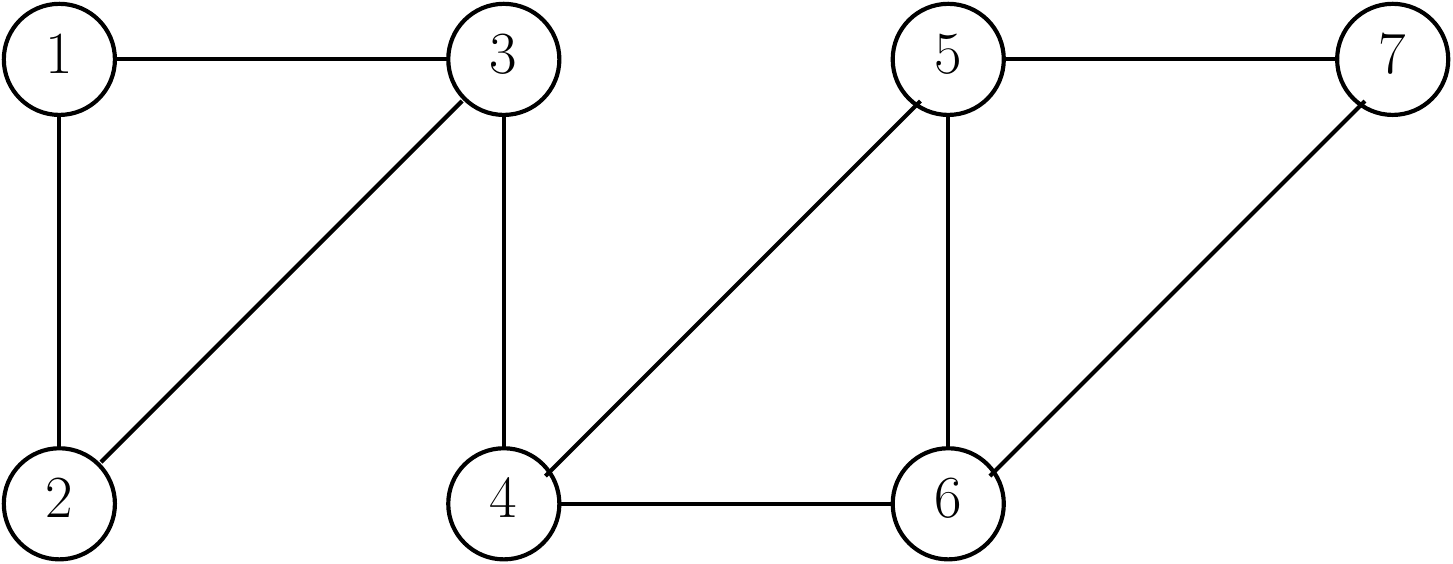}
\caption{\label{fig:example-graph}The graph $G$ of Example \ref{ex:graph}.}
\end{figure}
Consider the sandpile configuration
$c'=(0,0,2,0,2,2,*)$ on $G$. Applying Definition~\ref{def_psi} one
finds that $A=\{1,2\}$ is the minimal non-empty subset (w.r.t.
$\prec$) of $\{1,2,\ldots,7\}$ such that $c'+\Delta_A \in
\Stable(G)$. Thus we have
\begin{eqnarray*}
\psi(c') &=& c'+\Delta_1 + \Delta_2 \;=\; (0,0,2,0,2,2,*) + (2,-1,-1,0,0,0,*) + (-1,2,-1,0,0,0,*) \\
&=& (1,1,0,0,2,2,*).
\end{eqnarray*}
Next, consider the configuration $c''=(1,1,0,2,0,0,*)$. Applying
Definition~\ref{def_psi} we find that the minimal subset $A$ for
which $c''+\Delta_A$ is stable is $A=\{5,6\}$. Therefore
\begin{eqnarray*}
\psi(c'') &=& c''+\Delta_5 + \Delta_6 \;=\; (1,1,0,2,0,0,*) + (0,0,0,-1,3,-1,*) + (0,0,0,-1,-1,3,*)\\
&=& (1,1,0,0,2,2,*).
\end{eqnarray*}
Finally, since $\psi(c')=\psi(c'')$ but $c' \neq c''$, the
operator $\psi$ is not injective {\bf{even}} if restricted to
non-recurrent configurations. }}
\end{Ex}

We conclude this section by showing how operators $\varphi$ and $\psi$
are {\em conjugate}. Let us introduce a well-known involution $\beta$
defined on any configuration $c$ whose $i$-th component is
$$\beta(c)_i=(d_i-1)-c_i.$$
\begin{Rem}\label{rem:beta}
$\beta$ maps non-negative configurations to semi-stable
configurations. Of course $\beta$ is also an involution on stable
configurations. Moreover, this mapping induces a bijection from
parking to recurrent configurations \emph{that does not preserve
classes}.
\end{Rem}

\begin{Prop}\label{prop:conjugate}
One has the following relations between operators $\varphi$,
$\psi$ and $\beta$:
$$\varphi\cdot \beta=\beta\cdot \psi.$$
\end{Prop}

\begin{proof}
This proposition readily comes from the observation that for any
subset $A$ of $ \{1,2,\dots,n\}$
$$\beta(c+\Delta_A)=\beta(c)-\Delta_A$$
which implies that $c+\Delta_A$ is stable if and only if
$\beta(c)-\Delta_A$ is also stable.
\end{proof}

\section{A cyclic lemma counting parallelogram polyominoes in a $m\times n$ rectangle}\label{sec:cyclic_lemma}

The aim of this section is to present and prove the so called
{\it{Cyclic Lemma}} (Lemma~\ref{theo:cyclic_lemma}) for pairs of
paths in the plane. This lemma tells us how pairs of infinite
paths in the plane, which are formed from two binary words having
almost identical parameters in terms of the fixed global
parameters $m$ and $n$, can be partitioned into different classes
with respect to some points of intersection. The lemma shows us
that each of these classes have exactly the same size. Further to
this, we prove that every class corresponds to a unique
parallelogram polyomino whose bounding box is an $m\times n$
rectangle.

We first introduce some terminology pertinent to the remainder of the paper.
This terminology will be illustrated in the example of Subsection~\ref{THEEXAMPLE}.
Some readers may prefer to skip directly to this example.

An {\it{$(m,n)$-binomial word}} is any word $w$ over the alphabet
$\{N,E\}$ consisting of $m$ letter $E$'s and $n$ letter $N$'s. We
let $B_{m,n}$ be the set of all $(m,n)$-binomial words, of which
there are ${m+n\choose n}$ many. A vertex
$x=(x_1,x_2)\in\mathbb{Z}^2$ of the square lattice together with a
binomial word $w$ defines a path $\mypath{x}{w}$: this paths
starts at $x$ and is made up of unit steps given by the letters of
$w$ wherein $N$ corresponds to a north step $(0,1)$ and $E$ to an
east step $(1,0)$. By abuse of notation we will use the terms step
$\leftrightarrow$ letter and path $\leftrightarrow$ word
interchangeably. For a step $s$ in a path, we denote by
$(X_1(s),X_2(s))$ the coordinates of the starting vertex of this
step. We will sometimes index, non-ambiguously, the steps of a
binomial path: a step $N$ is called $N_i$ where $i$ is $X_1(N)$,
the ordinate of the starting vertex of this step, and a step $E$
is called $E_j$ where $j=X_2(E)$, the abscissa of the starting
vertex of this step. In a path $\mypath{x}{w}$, we let
$\mypathfactor{y}{w}{k}$ be the factor of this path that starts
from vertex $y$ in $\mypath{x}{w}$ and consists of the $k$ steps
that follow it in $\mypath{x}{w}$.

Employing this terminology to define parallelogram polyominoes, we
have the following: A polyomino $P$ is a {\emph{$m \times n$
parallelogram polyomino}} iff it is the set of unit cells of a
square lattice enclosed by a pair $(u,v)$ of $(m,n)$-binomial
paths intersecting only at their endpoints where
$u=\mypath{(0,0)}{Nu''E}$ and $v=\mypath{(0,0)}{Ev''N}$. If one
adds a final red east step to the red path in
Figure~\ref{fig:phi0} then the enclosed region is a $4 \times 6$
parallelogram polyomino. Let $\Para_{m,n}$ be the set of all
parallelogram polyominoes in an $m\times n$ rectangle, i.e. having
an $m\times n$ bounding box.

The number of primes on a sub-path of a path helps us to remember
the number of deleted letters/steps. For example $u''$ means that
two steps (or letters) have been removed from the path $u$, and
$\ell'$ means that only one letter has been removed from $\ell$.

In this definition, the steps $N_0$ and $E_{m-1}$ in $u$ and $E_0$
and $N_{n-1}$ are forced so that counting the polyominoes of
$\Para_{m,n}$ is equivalent to counting the number of
non-intersecting pairs of $(m-1,n-1)$-binomial paths
$(\mypath{(1,0)}{u''},\mypath{(0,1)}{v''})$ which end at positions
$((m-1,n),(m,n-1))$. Let us note that $u''=
\mypathfactor{(1,0)}{u}{m+n-2}$. This remark allows us to
recognize the framework of the classical LGV-lemma (see
\cite{LGV}) and then to count the polyominoes via the following
determinant:
$$|\Para_{m,n}|=\left|\begin{array}{cc}
    {m+n-2\choose m-1} & {m+n-2\choose m} \\[1em]
    {m+n-2\choose m-2} & {m+n-2 \choose m-1} \end{array}\right|
= \frac{1}{m}{m+n-2\choose m-1}{m+n-1\choose m-1}.$$

The rightmost expression in the preceding equation bears a
resemblance to the Catalan numbers $\frac{1}{2n+1}{2n+1\choose
n}$. Catalan numbers count the number of Dyck words of semi-length
$n$. The Dvoretzky-Motzkin (\cite{DVOR}) proof that the number of
Dyck words is given by the Catalan numbers involved using a cyclic
lemma that acted on partitions of all binomials words in
$B_{n,n+1}$, and showed that there was precisely one Dyck word
with an additional final east step in each of the partitions.
Every part of the partition had $2n+1$ elements, and this explains
the factor $1/(2n+1)$ in the expression for the Catalan numbers.

Their approach, combined with the similarity of the expressions,
suggests it may be possible to employ similar machinery in our
setting, and therefore reprove the expression for $|\Para_{m,n}|$
above. Indeed, this is exactly what we will do next to a partition
$\Pi_{m,n}$ of non-constrained pairs of binomial paths in
$B_{m-1,n-1}\times B_{m-1,n}$.

\subsection{An example illustrating the terminology}\label{THEEXAMPLE}

The example given in this subsection illustrates the terminology
of this section. In the example we compute a part $\pi(u'',\ell')$
in our Cyclic Lemma (Lemma~\ref{theo:cyclic_lemma}). There is some
yet-to-be-explained information contained in these figures, since
this example will later allow us to read the iterates of the
operator $\varphi$ of a configuration on $K_{4,6}$ which will run
from the recurrent to the parking configuration of all stable
configurations of a toppling class.

In this example $m=4$ and $n=6$ and we choose the binomial paths
$u''=ENNENENN\in B_{3,5}$ and $\ell'=NNNEENENN\in B_{3,6}$. We
draw a factor of two yet-to-be defined red and green bi-infinite
periodic paths. The red path $(Nu'')^\mathbb{Z}$ contains the
factor $[Nu''.Nu''.Nu'']_{(0,0)|23}$ that starts from the origin
$(0,0)$ and is made up of $23$ steps. Similarly the green path
$(E\ell')^\mathbb{Z}$ contains the factor
$[E\ell'.E\ell'.E\ell']_{(0,0)|24}$. The origin $z^{(3)}=(0,0)$ is
marked with an orange disk which indicates that it is the bottom
left corner of an $m\times n$ rectangle, also drawn in orange, and
called the {\it{$z^{(3)}$-rectangle}}.
\figureornotfigure{
\begin{figure}[H]
\begin{tikzpicture}[scale=0.7]
\draw[draw=white!50!black]  (5, 9) grid (9, 15);
\draw[rounded corners, opacity=0.3, line width=5, draw = orange] (5, 9) rectangle (9, 15);
\draw[line width=2,draw=red] (4.95, 9.05)--(4.95, 10.05)--(5.95, 10.05)--(5.95, 13.05)--(6.95, 13.05)--(6.95, 15.05)--(7.95, 15.05)--(7.95, 15.05);
\draw[line width=2,draw=dgreen] (5.05, 8.95)--(7.05, 8.95)--(7.05, 9.95)--(8.05, 9.95)--(8.05, 10.95)--(8.05, 10.95)--(8.05, 11.95)--(9.05, 11.95)--(9.05, 12.95)--(9.05, 12.95)--(9.05, 13.95)--(9.05, 13.95)--(9.05, 14.95);
\draw[fill=orange] (5, 9) circle (0.2);
\draw[orange] node at (5+0.4,9-0.4) {$z^{(0)}$};
\foreach \x/\ux in {5/9,6/9,7/10,8/12}{
\draw[dgreen] node at (\x+0.6,\ux+0.3) {\small$E_{\x}$};
}
\foreach \x/\ux in {5/9,6/10,6/11,6/12,7/13,7/14}{
\draw[red] node at (\x+0.4,\ux+0.7) {\small $N_{\ux}$};
}
\draw[draw=white!50!black]  (3, 4) grid (7, 10);
\draw[rounded corners, opacity=0.3, line width=5, draw = orange] (3, 4) rectangle (7, 10);
\draw[line width=2,draw=red] (2.95, 4.05)--(2.95, 7.05)--(3.95, 7.05)--(3.95, 9.05)--(4.95, 9.05)--(4.95, 10.05)--(5.95, 10.05)--(5.95, 10.05);
\draw[line width=2,draw=dgreen] (3.05, 3.95)--(4.05, 3.95)--(4.05, 4.95)--(4.05, 4.95)--(4.05, 5.95)--(5.05, 5.95)--(5.05, 6.95)--(5.05, 6.95)--(5.05, 7.95)--(5.05, 7.95)--(5.05, 8.95)--(7.05, 8.95)--(7.05, 9.95);
\draw[fill=orange] (3, 4) circle (0.2);
\draw[orange] node at (3+0.4,4-0.4) {$z^{(1)}$};
\foreach \x/\ux in {3/4,4/6,5/9,6/9}{
\draw[dgreen] node at (\x+0.6,\ux+0.3) {\small$E_{\x}$};
}
\foreach \x/\ux in {3/4,3/5,3/6,4/7,4/8,5/9}{
\draw[red] node at (\x+0.4,\ux+0.7) {\small $N_{\ux}$};
}
\draw[draw=white!50!black]  (2, 3) grid (6, 9);
\draw[rounded corners, opacity=0.3, line width=5, draw = orange] (2, 3) rectangle (6, 9);
\draw[line width=2,draw=red] (1.95, 3.05)--(1.95, 4.05)--(2.95, 4.05)--(2.95, 7.05)--(3.95, 7.05)--(3.95, 9.05)--(4.95, 9.05)--(4.95, 9.05);
\draw[line width=2,draw=dgreen] (2.05, 2.95)--(3.05, 2.95)--(3.05, 3.95)--(4.05, 3.95)--(4.05, 4.95)--(4.05, 4.95)--(4.05, 5.95)--(5.05, 5.95)--(5.05, 6.95)--(5.05, 6.95)--(5.05, 7.95)--(5.05, 7.95)--(5.05, 8.95);
\draw[fill=orange] (2, 3) circle (0.2);
\draw[orange] node at (2+0.4,3-0.4) {$z^{(2)}$};
\foreach \x/\ux in {2/3,3/4,4/6,5/9}{
\draw[dgreen] node at (\x+0.6,\ux+0.3) {\small$E_{\x}$};
}
\foreach \x/\ux in {2/3,3/4,3/5,3/6,4/7,4/8}{
\draw[red] node at (\x+0.4,\ux+0.7) {\small $N_{\ux}$};
}
\draw[draw=white!50!black]  (0, 0) grid (4, 6);
\draw[rounded corners, opacity=0.3, line width=5, draw = orange] (0, 0) rectangle (4, 6);
\draw[line width=2,draw=red] (-0.05, 0.05)--(-0.05, 1.05)--(0.95, 1.05)--(0.95, 3.05)--(1.95, 3.05)--(1.95, 4.05)--(2.95, 4.05)--(2.95, 6.05);
\draw[line width=2,draw=dgreen] (0.05, -0.05)--(1.05, -0.05)--(1.05, 0.95)--(1.05, 0.95)--(1.05, 1.95)--(1.05, 1.95)--(1.05, 2.95)--(3.05, 2.95)--(3.05, 3.95)--(4.05, 3.95)--(4.05, 4.95)--(4.05, 4.95)--(4.05, 5.95);
\draw[fill=orange] (0, 0) circle (0.2);
\draw[orange] node at (0+0.4,0-0.4) {$z^{(3)}$};
\foreach \x/\ux in {0/0,1/3,2/3,3/4}{
\draw[dgreen] node at (\x+0.6,\ux+0.3) {\small$E_{\x}$};
}
\foreach \x/\ux in {0/0,1/1,1/2,2/3,3/4,3/5}{
\draw[red] node at (\x+0.4,\ux+0.7) {\small $N_{\ux}$};
}
\end{tikzpicture}
\caption{The relevant part of $((Nu'')^\mathbb{Z},(E\ell')^\mathbb{Z})$ to compute $\pi(u'',\ell')$ and the iterations of $\varphi^k(u)$ for binomials paths $u''=ENNENENN$, $\ell'=NNNEENENN$ and configuration $c={1,2,2,3,3,3,\choose 0,3,5,*}$ on $K_{4,6}$.\label{fig:infinite-pair}}
\end{figure}
} 

Notice that, by definition, removing the first north red step
$N_0$ of the factor made up of red steps included in the
$z^{(3)}$-rectangle gives the path $u''$. Similarly, removing the east
step $E_0$ of the green factor included in the $z^{(3)}$-rectangle
gives the path $\ell'$. In this example, there are $m=4$ orange
vertices $(z^{(i)})_{i=0,\ldots ,3}$ which correspond to stable
intersections, i.e. those intersections of the red and green paths
which are continued with a red north step and green east step.

A key property of our Cyclic Lemma is that there are exactly $m$
stable intersections for any choice of $u''$ and $\ell'$. As for
the $z^{(3)}$-rectangle, we can extract from each $z^{(i)}$-rectangle a
pair of binomials paths in $B_{3,5}\times B_{3,6}$. This is done
by deleting the first red step and first green step in the
$z^{(i)}$-rectangle. These $z^{(i)}$-rectangles are also illustrated in
Figure~\ref{fig:toppling-class}.

The proof of our Cyclic Lemma relies on the key parameter
$\pos(E_i)$. This parameter is defined for every east step of the
green path of Figure~\ref{fig:infinite-pair}. We describe it here
geometrically so that the reader may bypass the formal symbolic
definition: $\pos(E_8)=8-6=2$ because the starting point of step
$E_8$ has abscissa $8$ and the starting point of step $N_{12}$
(chosen because this is the unique north step which has the same
ordinate as $E_8$) has abscissa $6$.

An equivalent way to define the orange stable intersections $z^{(i)}$
is: a point is an orange stable intersection if it is the starting
point of east green step $E_j$ such that $\pos(E_j)=0$. In a
$z^{(i)}$-rectangle, the region between the lines is a parallelogram
polyomino if and only if $\pos(E_j)>0$ for every east green $E_j$
in the $z^{(i)}$-rectangle, except the first $E_k$ for which
$\pos(E_k)=0$. In the example the parallelogram polyomino is in
the $z^{(0)}$-rectangle.

Notice that any green east step $E_i$ in the $z^{(3)}$-rectangle
satisfies $\pos(E_i)\leq 0$. Theorem~\ref{thm55} will show that
moving from the $z^{(i)}$-rectangle to the $z^{(i+1)}$-rectangle is
equivalent to one application of the operator $\varphi$ to a
stable configuration on $K_{m,n}$.

A final remark: the orange $z^{(i)}$ vertices are defined using the
green east steps and the red north steps.  We call such steps
{\it{frame steps}}. Green north steps and red east steps are used
to define the sorted stable configurations on $K_{m,n}$ so we call
those the {\it{configurations steps}}.

\subsection{Partitioning paths and a Cyclic Lemma}

\label{sec:cycliclemma} We now define the partition $\Pi_{m,n}$.
This definition relies on some pairs of periodic bi-infinite
paths. The elements of a generic part $\pi_k$ will be exactly
those for which the pairs of bi-infinite paths differ only by a
geometric translation. Given an $(m,n)$-binomial word $w$, we
define the bi-infinite path $w^\mathbb{Z}$ as the concatenation of
the infinite sequence of paths
$(\mypath{(mi,ni)}{w})_{i\in\mathbb{Z}}$. This concatenation is
well-defined since the last vertex of $\mypath{(mi,ni)}{w}$ is
$(mi+m,ni+n)$ which is the starting vertex of
$\mypath{(m(i+1),n(i+1))}{w}$.

We define the bi-infinite pair of a pair $(u'',\ell')\in
B_{m-1,n-1}\times B_{m-1,n}$ to be
$((Nu'')^\mathbb{Z},(E\ell')^\mathbb{Z})$.  A vertex $x$ at an
intersection of $(Nu'')^\mathbb{Z}$ and $(E\ell')^\mathbb{Z}$ is
called a {\it{stable intersection}} if
$\mypathfactor{x}{(Nu'')^\mathbb{Z}}{1}=\mypath{x}{N}$ and
$\mypathfactor{x}{(E\ell')^\mathbb{Z}}{1}=\mypath{x}{E}$, i.e. the
intersection is followed by a north step in $(Nu'')^\mathbb{Z}$
and an east step in $(E\ell')^\mathbb{Z}$.  Instead of using the
equivalence by translation, we use the stable intersection to
define the part $\pi_{(u'',\ell')}$ of a generic pair
$(u'',\ell')$ in the partition $\Pi_{m,n}$:
$$ \pi_{(u'',\ell')} =
\left\{\left(\mypathfactor{(y_1,y_2+1)}{(Nu'')^\mathbb{Z}}{m+n-2},
    \;\;\mypathfactor{(y_1+1,y_2)}{(E\ell')^\mathbb{Z}}{m+n-1}\right)\right\}$$
where $y=(y_1,y_2)$ runs over all stable intersections of
$((Nu'')^\mathbb{Z},(E\ell')^\mathbb{Z})$. In other words, we are
considering the factors of a path just after the stable
intersection which follows the forced initial steps.
\begin{Lem}[Cyclic Lemma]
\label{theo:cyclic_lemma} The well-defined partition
$\Pi_{m,n}=\bigcup_{k}\pi_k$ of all pairs $(u'',\ell')$ that are
made up from an $(m-1,n-1)$-binomial path $u''$ and an
$(m-1,n)$-binomial path $\ell'$ satisfies:
\begin{itemize}
\item The cardinality $|\pi_k|$ of every part $\pi_k$ is $m$.
\item In each part $\pi_k$ there is exactly one pair $(u'',\ell')$ such
  that $\left(\mypath{(0,0)}{Nu''E},\mypath{(0,0)}{E\ell'}\right)$ describes a
  polyomino in $\Para_{m,n}$.
\end{itemize}
\end{Lem}

The enumeration of polyominoes in $\Para_{m,n}$ is an immediate
corollary: the pairs $(u'',\ell')$ are an interpretation of
${m+n-2 \choose m-1}{m+n-1\choose m-1}$ and the properties of the
partition allow one to select one pair for every part $\pi_k$,
i.e. divide the total number by $m$. This partition is different
to the one given in Huq~\cite[3.1.2]{AH} but the same
as the one given in Chottin~\cite{CHO75}. In addition, it has a
deep relation with an algorithm computing the operator $\varphi$
on $K_{m,n}$ which we study in the following sections. Aval et
al.~\cite{ABG} use this cyclic lemma to calculate the Frobenius
characteristic of the action of the symmetric group on labelled
parallelogram polyominoes.

\begin{proof}[Proof of Lemma~\ref{theo:cyclic_lemma}]
We show that the binary relation $R$ on pairs of words in
$B_{m-1,n-1}\times B_{m-1,n}$ defined by
$$ (v'',k') R (u'',\ell') \iff (v'',k') \in \pi_{(u'',\ell')} $$
is an equivalence relation whose classes are the parts of $\Pi_{m,n}$
and then that each class contains exactly $m$ elements.
First we will verify the three defining properties of an equivalence relation.

\begin{description}
\item[Reflexivity] By definition of
$((Nu'')^{\mathbb{Z}},(E\ell')^{\mathbb{Z}})$, point $(0,0)$ is a
stable intersection and so
$$ \left(\mypathfactor{(0,1)}{(Nu'')^\mathbb{Z}}{m+n-2},
  \;\;\mypathfactor{(1,0)}{(E\ell')^\mathbb{Z}}{m+n-1}\right)= (u'',\ell')$$
belongs to $\pi_{(u'',\ell')}$. Therefore $R$ is reflexive.

\item[Symmetry] Let $(v'',k') \in \pi_{(u'',\ell')}$ which, by
definition, occurs as factors from a stable intersection
$(y_1,y_2)$.  Since $Nv''$ and $Nu''$ describe, up to some cyclic
conjugate, exactly the complete periodic pattern of
$(Nu'')^\mathbb{Z}$ we remark that $(Nu'')^\mathbb{Z}$ is exactly
image of $(Nv'')^\mathbb{Z}$ by the vector translation
$(-y_1,-y_2)$ which sends the stable intersection $(y_1,y_2)$ to
$(0,0)$. We have exactly the same relation for
$(E\ell')^\mathbb{Z}$ and $(Ek')^\mathbb{Z}$ so
$((Nv'')^\mathbb{Z},(Ek')^\mathbb{Z})$ is the image of
$((Nu'')^\mathbb{Z},(E\ell')^\mathbb{Z})$ by the vector
translation $(-y_1,-y_2)$.  This implies, in particular, that as
the image of the stable intersection $(0,0)$ in
$((Nu'')^\mathbb{Z},(E\ell')^\mathbb{Z})$ defining $(u'',\ell')$,
the vertex $(-y_1,-y_2)$ is a stable intersection in
$((Nv'')^\mathbb{Z},(Ek')^\mathbb{Z})$ from which we deduce that
$(u'',\ell')\in \pi_{(v'',k')}$. Therefore $R$ is symmetric.
\item[Transitivity] These translations between the bi-infinite
pairs of paths also imply transitivity of $R$. We consider three
pairs of words in $B_{m-1,n-1}\times B_{m-1,n}$: $(u'',\ell')$,
$(v'',k')$ and $(w'',j')$ such that $(u'',\ell')\in
\pi_{(v'',k')}$ and $(v'',k')\in \pi_{(w'',j')}$.  More precisely,
$(u'',\ell')$ appears in $((Nv'')^\mathbb{Z},(Ek')^\mathbb{Z})$
from the stable intersection $(x_1,x_2)$ and $(v'',k')$ appears in
$((Nw'')^\mathbb{Z},(Ej')^\mathbb{Z})$ from the stable
intersection $(y_1,y_2)$. The translations between the bi-infinite
paths imply that $(u'',\ell')$ appears in
$((Nw'')^\mathbb{Z},(Ej')^\mathbb{Z})$ from the stable
intersection $(x_1+y_1,x_2+y_2)$. Therefore $R$ is transitive.
\end{description}

The equivalence classes of $R$ are described by the parts $\pi_k$
of the well-defined partition $\Pi_{m,n}$. We have three things
left to show:
\begin{enumerate}
\item[(i)] {\it{Every part $\pi_k$ of $\Pi_{m,n}$ contains exactly
$m$ elements:}} Our proof that each part $\pi_k$ contains exactly
$m$ elements relies on the following key parameter.  For any step
$E_i$ in $(E\ell')^\mathbb{Z}$ we define its relative position
$\pos(E_i)$ to be the only step $N_j$ in $(Nu'')^\mathbb{Z}$,
where $j=X_2(E_i)$, that may start from the common stable
intersection: $$ \pos(E_i) =
X_1(E_i)-X_1\left(N_{X_2(E_i)}\right).$$ Since $E\ell'$ contains
exactly one more east step than $Nu''$ we have the relation
$$ \pos(E_{i+m}) = \pos(E_i)+1.$$
This shows that the equation $\pos(E_{mi+k})=0$ defining a stable
intersection admits exactly one solution for each $k=0,1,\ldots,
m-1$, thereby giving the $m$ stable intersections from which we
extract the factors giving the element of $\pi_{(u'',\ell')}$.
\item[(ii)] {\it{Each of the extracted elements are distinct:}} To
show this we introduce a strictly increasing parameter on the
frame east steps: the cumulated relative position
$\cumuledpos(E_j)$ of a step $E_j$ in $(E\ell')^\mathbb{Z}$ is
$$\cumuledpos(E_j) = \sum_{k=0}^{m-1} \pos(E_{j+k}).$$
From the previous relation between $\pos(E_{i+m})$ and
$\pos(E_i)$, we have
$$ \cumuledpos(E_{j+1})=\cumuledpos(E_{j})+1.$$
For each stable intersection followed by the east frame step
$E_i$, $\cumuledpos(E_i)$ is also a function of the extracted pair
of paths $(u'',\ell')$ since this parameter can be computed using
$Nu''E$ and $E\ell'$.  Two different stable intersections which
are followed by $E_i$ and $E_j$, where $i<j$, lead to two distinct
pairs since $\cumuledpos(E_i) \neq \cumuledpos(E_j)$.

\item[(iii)] {\it{}} The solitary parallelogram polyomino in a
class $\pi_k$ is also extracted via the relative position of the
east frame steps. Let $E_j$ be the maximal $j\in\mathbb{Z}$ such
that $\pos(E_j)=0$. This defines the ``last'' stable intersection
$y$.  By choice of $E_j$, $\pos(E_{j+k})>0$ for $k=1,\ldots, m$
and this shows that the factor
$\mypathfactor{y}{(E\ell')^\mathbb{Z}}{m+n}$ is below
$\mypathfactor{y}{(Nu'')^\mathbb{Z}}{m+n-1}$ and intersects it
only at $y$, whereas this is not the case for any other stable
intersection. Thus only this pair of factors defines a polyomino
in this part $\pi_{(u'',\ell')}$ by adding an final east step to
the shorter path. For any other stable intersection $z=(z_1,z_2)$,
the fact that $\pos(E_j)\leq 0$ for some east step of
$\mypathfactor{(z_1+1,z_2)}{(E\ell')^\mathbb{Z}}{m+n-1}$ implies
that $\mypathfactor{z}{(E\ell')^\mathbb{Z}}{m+n}$ intersects
$\mypathfactor{z}{(Nu'')^\mathbb{Z}}{m+n-1}$ outside their
endpoints and does not define a parallelogram polyomino.\qedhere
\end{enumerate}
\end{proof}

\section{An algorithm to compute $\varphi$ for stable configurations on $K_{m,n}$.}
\label{sec:algo-phi}

In this section we will give an algorithm which computes $\varphi$
for stable configurations on $K_{m,n}$. On $K_{m,n}$, the edge set
consists of single edges between vertices $v_i$ and $v_j$ such
that $i\leq n < j$. Thus the vertex set $V=\{v_1,\ldots,
v_{n+m}\}$ may be split into two sets that we, albeit abusively,
call the {\it{non-sink component}}, $C^{\leq n}_{m,n} =
\{v_1,\ldots , v_{n}\}$, and the {\it{sink component}},
$C^{>n}_{m,n} = \{v_{n+1},\ldots , v_{m+n}\}$, since the sink is
$v_{n+m}$. Let $c$ be a generic configuration on $K_{m,n}$. The
(partial) non-sink configuration $c^{\leq n}$ is $(c_i)_{1\leq
i\leq n}=(c_1,\ldots,c_n)$, the restriction of $c$ to the non-sink
component. The (partial) sink configuration $c^{>n}$ is
$(c_i)_{n<i<n+m}$, the restriction of $c$ to the sink component
but \emph{excluding the sink} $v_{n+m}$.

Due to the symmetries of $K_{m,n}$ with its distinguished vertex
$v_{n+m}$, it is natural to consider the two symmetric group
actions $S_n$ and $S_{m-1}$ on the non-sink configurations and
sink configurations, respectively:
\begin{align*}
\sigma.c^{\leq n} &= \left(c_{\sigma(i)}\right)_{1\leq i\leq n}&& \mbox{ for every }\sigma \in S_{n}\\
\tau.c^{>n} &= \left(c_{\tau(i)}\right)_{n<i<n+m} && \mbox{ for every }\tau \in S_{m-1}.
\end{align*}
We will call a configuration $c$ \emph{sorted} if both its sink
and non-sink configurations are weakly increasing: $c_1 \leq c_2
\leq \ldots \leq c_n$ and $c_{n+1}\leq c_{n+2}\leq \ldots \leq
c_{n+m-1}$. Under the action of $S_{n}\times S_{m-1}$ given above,
the computations are equivalent up to permutations of this group.
Without loss of generality, we will henceforth work at the level
of orbits using the sorted configurations as representatives.

The interaction between permutations and toppling at this level of
orbits suggests the introduction of toppling and permuting
equivalence. Two configurations $u$ and $v$ are \emph{toppling and
permuting equivalent} if there exists a finite sequence of
topplings followed by the action of a permutation which turns $u$
into $v$. The sorted recurrent configuration are canonical
elements of the classes of this equivalence.

We consider topplings which start from stable configurations and
preserve the following helpful assumption. A configuration $c$
satisfies the \emph{compact range assumption} if
$$\max(c^{\leq n}) - \min(c^{\leq n}) \leq m\quad\mbox{ and }\quad\max(c^{>n}) - \min(c^{>n}) \leq n,$$
where $\max(c^{\leq n})= \max(c_1,\ldots,c_n)$ etc. Given a sorted
configuration $c$, let $T^{\leq n}(c)$ be the result of first
toppling in the non-sink component by toppling $v_n$, and then
sorting all entries in the non-sink component so that the
resulting configuration is once again sorted. Similarly, given a
sorted configuration $c$, let $T^{>n}(c)$ be the result of first
toppling in the sink component by toppling $v_{m+n-1}$ and then
sorting all entries in the sink component so that the resulting
configuration is once again sorted.

\begin{Lem}
\label{lem:operator-on-compact-range}
If a sorted configuration $c$ satisfies the compact range assumption, then
  $$ T^{\leq n}(c) = (c_n-m,c_1,\ldots, c_{n-1},1+c_{n+1},\ldots, 1+c_{n+m-1})$$
  and
  $$ T^{>n}(c) = (1+c_1,\ldots ,1+c_n,c_{n+m-1}-n,c_{n+1},\ldots,c_{n+m-2}),$$
  both of which satisfy the compact range assumption.
\end{Lem}
\begin{proof}
  The toppling of the vertex $v_n$ in configuration $c = (c_1,\ldots
  c_{n},c_{n+1},\ldots , c_{n+m-1})$ leads to the configuration
  $c'=(c_1,\ldots,c_{n-1},c_n-m,c_{n+1}+1,\ldots,c_{n+m-1}+1)$.  Since
  $c$ satisfies the compact range assumption, $c_n-c_1\leq m$ is equivalent to
  $c_n-m\leq c_1$ and so $(c_n-m,c_1,\ldots, c_{n-1},1+c_{n+1},\ldots
  1+c_{n+m-1})$ is the sorted configuration representing the orbit of
  $c'$ which, by definition, is $T^{\leq n}(c)$.

Proof of the expression for $T^{>n}(c)$ is analogous and differs
only in using the compact range assumption $c_{m+n-1}-c_{n+1}\leq
n$.
\end{proof}

In the case of $K_{m,n}$, the operator $\varphi$ has some additional
regularities on sorted configuration that we describe in the following
proposition.
\begin{Prop}
  \label{prop:two-parameter-description-of-A}
  Let $c$ be a sorted stable configuration on $K_{m,n}$.  The minimal
  set $A$, if it exists, which defines $\varphi(c) = c-\Delta_A$ is
$$ A = \{v_{n-k},\ldots , v_n\}\cup\{v_{m+n-1-l},\ldots ,v_{n+m-1}\}$$
for some $0\leq k \leq n-1$ and $0\leq l \leq m-2$. In this event,
$c_{n-k-1} < c_{n-k}$ and $c_{m+n-2-l} < c_{m+n-1-l}$.
\end{Prop}

This proposition is deduced from the following two lemmas:
\begin{Lem}\label{lem:segment}
Suppose two vertices $v_i$ and $v_j$ are in the same component of $K_{m,n}$.
Further suppose that $A$ is as in Definition~\ref{def_phi}.
If $c_i\geq c_j$, then ($v_j\in A \implies v_i \in A$).
\end{Lem}

\begin{proof}
 Let $d$ denote the degree of vertices in the component under consideration
  and let $t$ be the number of vertices in the intersection of $A$ and the other component.  If
  $v_j\in A$, then $(\varphi(c))_j = c_j+t-d \geq 0$ since $\varphi(c)$ is
  non-negative. If $v_i\notin A$, then $d \leq c_j+t \leq c_i+t =
  (\varphi(c))_i$ which gives a contradiction since $\varphi(c)$ is stable
  (and $(\varphi(c))_i< d$) and so $v_i$ must be in $A$.
\end{proof}

\begin{Lem}\label{lem:non-empty}
Let $c$ be a sorted stable configuration on $K_{m,n}$.
Suppose that the set $A$ in Definition~\ref{def_phi} exists.
Then both $v_n$ and $v_{n+m-1}$ are members of $A$.
\end{Lem}

\begin{proof}
  Since $A$ is non-empty, from Lemma~\ref{lem:segment} $v_n\in A$ or
  $v_{n+m-1}\in A$. Since the proof is symmetric in the two cases,
  assume without loss of generality that $v_n\in A$.
  As  $c$ is stable, $c_n \leq m-1$ and because
  $\varphi(c)$ is non-negative $(\varphi(c))_n=c_n-m+t\geq 0$ where $t$ is the
  number of toppled vertices in the sink component. These two
  inequalities imply that $t\geq 1$ so at least one vertex of the
  sink component belongs to $A$ which, by Lemma~\ref{lem:segment}, implies
  $v_{m+n-1}\in A$.
 \end{proof}

\begin{proof}[Proof of Proposition \ref{prop:two-parameter-description-of-A}]
From Lemma~\ref{lem:non-empty}, the two following intersections
$$ A_{\mathrm{non-sink}} = A \cap \{v_1,\ldots ,v_n\} \mbox{ and } A_{\mathrm{sink}} = A \cap \{v_{n+1},\ldots,v_{n+m-1}\}$$
are non-empty so let $v_i$, respectively $v_j$, be the vertex of minimal index of $A_{\mathrm{non-sink}}$, respectively $A_{\mathrm{sink}}$.

From Lemma~\ref{lem:segment} and the fact that $c$ is sorted we have $c_{i-1}<c_i$,
when $c_{i-1}$ exists, and since $c_j\geq c_i$ for all $1\leq j \leq n$ we also have
$$ A_{\mathrm{non-sink}} = \{v_i,v_{i+1},\ldots ,v_{n}\} = \{v_{n-k},\ldots , v_{n}\} $$
where $k = n-i$.
An analogous argument for $A_{\mathrm{sink}}$ gives
\begin{align*}
 A_{\mathrm{sink}} &= \{v_{n+m-1-l},\ldots ,v_{n+m-1}\}.\qedhere
\end{align*}
\end{proof}

 These results culminate in Algorithm~\ref{alg:phi} which computes
 $\varphi$ for any stable configuration. Some minor additional
 terminology is needed: The configuration $0$ is the configuration $c$ such
 that $c_i=0$ at every vertex. The configuration $\delta$ is the
 stable configuration with the maximal number of grains at every vertex, i.e.
 $\delta^{\leq n}_i=m-1$ and $\delta^{>n}_i=n-1$ for all vertices
 $v_i$. Two (partial) configurations $c$ and $q$ satisfy $c\lesquare q$
 if $c_i\leq q_i$ for all $i$.

 The variable $nloops$ counts the number of loop iterations and is used
 exactly when $v$ is a parking configuration, since then the algorithm
 loops endlessly if we do not stop it as we do.

\begin{algorithm}
\caption{An algorithm that computes $\varphi(c)$ for stable configurations
$c$ on $K_{m,n}$}\label{alg:phi}
\begin{algorithmic}[1]
\Procedure{$\varphi$}{$c$}
   \State $c'\gets T^{\leq n} \cdot T^{>n}(c)$
   \State $nloops\gets 0$
   \While{not($0\lesquare c' \lesquare \delta$)}
    \State {\bf if} $nloops \geq m+n$ {\bf then} {\bf return $c$} {\bf end if}
    \State {\bf if} $c'^{\leq n}\nlesquare \delta^{\leq n}$ or $0^{>n}\nlesquare c'^{>n}$ {\bf then}  $c'\gets T^{\leq n}(c')$ {\bf end if}
    \State {\bf if} $c'^{> n}\nlesquare \delta^{> n}$ or $0^{\leq n}\nlesquare c'^{\leq n}$ {\bf then} $c'\gets T^{> n}(c')$ {\bf end if}
    \State $nloops\gets nloops+1$
   \EndWhile\label{euclidendwhile}
   \State \textbf{return} $c'$
\EndProcedure
\end{algorithmic}
\end{algorithm}

\section{Interpreting Algorithm~\ref{alg:phi} as a moving frame on pairs of paths}

In this section we bring together the results of the previous
sections. We will represent toppling and permuting equivalent classes
of sandpile configurations on $K_{m,n}$ as bi-infinite pairs of paths
in the plane, up to translation. The types of paths are precisely
those that were used in Section~\ref{sec:cyclic_lemma}.
Sorted configurations of the sandpile model may be read from these
bi-infinite paths by placing a `frame' at certain points of
intersection and performing measurements to steps of the paths from
this frame. We present and prove results which show how the
calculation of $\varphi$ on sandpile configurations in
Algorithm~\ref{alg:phi} may be interpreted as the moving of the frame
to a new point in the plane.

The underlying theme of this section is graphic in nature and we
encourage the reader to refer to the examples in the diagrams when
attempting to interpret the results.

Let $(u'',\ell')\in B_{m-1,n-1}\times B_{m-1,n}$ be a pair of
binomial words. We will assume that there are two bi-infinite
paths, $(Nu'')^{\ZZ}$ which is coloured red and $(E\ell')^{\ZZ}$
which is coloured green, in the plane. We will label half of these
steps as follows and refer to this labelling as a {\it{canonical
labelling}}.

\begin{itemize}
\item Label every $N$ step in $(E\ell')^{\ZZ}$ with $N_i$ where $i$ is the ordinate of the lower point of the step.
\item Label every $E$ step in $(Nu'')^{\ZZ}$ with $E_i$ where $i$ is the abscissa of the leftmost point of the step.
\end{itemize}
\figureornotfigure{
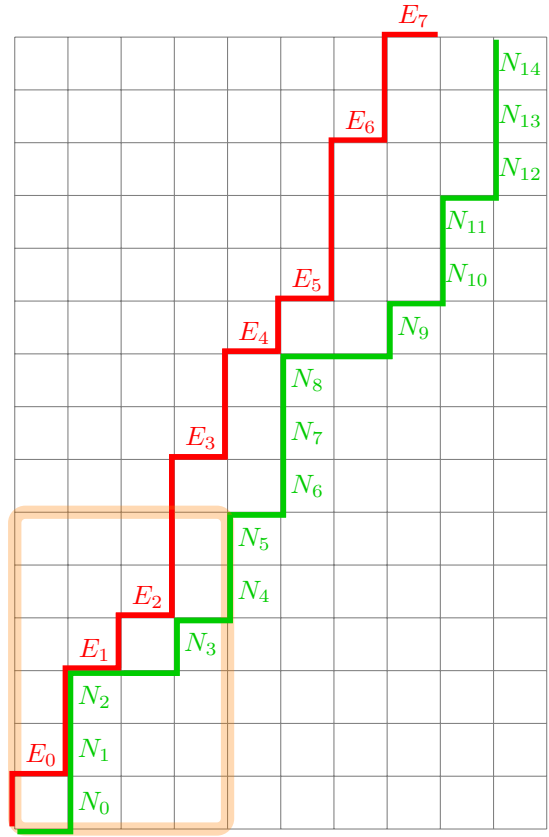
\begin{figure}[!h]
\begin{tikzpicture}[scale=0.7]
\draw[draw=white!50!black]  (0,0) grid (10, 15);
\draw[line width=2,draw=red] (4.95, 9.05)--(4.95, 10.05)--(5.95, 10.05)--(5.95, 13.05)--(6.95, 13.05)--(6.95, 15.05)--(7.95, 15.05)--(7.95, 15.05);
\draw[line width=2,draw=dgreen] (5.05, 8.95)--(7.05, 8.95)--(7.05, 9.95)--(8.05, 9.95)--(8.05, 10.95)--(8.05, 10.95)--(8.05, 11.95)--(9.05, 11.95)--(9.05, 12.95)--(9.05, 12.95)--(9.05, 13.95)--(9.05, 13.95)--(9.05, 14.95);
\foreach \x/\ux in {0/1,1/3,2/4,3/7,4/9,5/10,6/13,7/15}{
\draw[red] node at (\x+0.5,\ux+0.4) {\small$E_{\x}$};
}
\foreach \x/\ux in {1/0,1/1,1/2,3/3,4/4,4/5,5/6,5/7,5/8,7/9,8/10,8/11,9/12,9/13,9/14}{
\draw[dgreen] node at (\x+0.5,\ux+0.5) {\small $N_{\ux}$};
}
\draw[line width=2,draw=red] (2.95, 4.05)--(2.95, 7.05)--(3.95, 7.05)--(3.95, 9.05)--(4.95, 9.05)--(4.95, 10.05)--(5.95, 10.05)--(5.95, 10.05);
\draw[line width=2,draw=dgreen] (3.05, 3.95)--(4.05, 3.95)--(4.05, 4.95)--(4.05, 4.95)--(4.05, 5.95)--(5.05, 5.95)--(5.05, 6.95)--(5.05, 6.95)--(5.05, 7.95)--(5.05, 7.95)--(5.05, 8.95)--(7.05, 8.95)--(7.05, 9.95);
\draw[line width=2,draw=red] (1.95, 3.05)--(1.95, 4.05)--(2.95, 4.05)--(2.95, 7.05)--(3.95, 7.05)--(3.95, 9.05)--(4.95, 9.05)--(4.95, 9.05);
\draw[line width=2,draw=dgreen] (2.05, 2.95)--(3.05, 2.95)--(3.05, 3.95)--(4.05, 3.95)--(4.05, 4.95)--(4.05, 4.95)--(4.05, 5.95)--(5.05, 5.95)--(5.05, 6.95)--(5.05, 6.95)--(5.05, 7.95)--(5.05, 7.95)--(5.05, 8.95);
\draw[rounded corners, opacity=0.3, line width=5, draw = orange] (0, 0) rectangle (4, 6);
\draw[line width=2,draw=red] (-0.05, 0.05)--(-0.05, 1.05)--(0.95, 1.05)--(0.95, 3.05)--(1.95, 3.05)--(1.95, 4.05)--(2.95, 4.05)--(2.95, 6.05);
\draw[line width=2,draw=dgreen] (0.05, -0.05)--(1.05, -0.05)--(1.05, 0.95)--(1.05, 0.95)--(1.05, 1.95)--(1.05, 1.95)--(1.05, 2.95)--(3.05, 2.95)--(3.05, 3.95)--(4.05, 3.95)--(4.05, 4.95)--(4.05, 4.95)--(4.05, 5.95);
\end{tikzpicture}
\caption{ \label{fig:canon_label} The canonical labelling of the
pair of bi-infinite paths
$((Nu'')^\mathbb{Z},(E\ell')^\mathbb{Z})$ where $u''=ENNENENN$,
$\ell'=NNNEENENN$. The bottom left corner is the origin. Notice
that the steps that have been labelled are precisely those steps
of Figure~\ref{fig:infinite-pair} that were not labelled. Note
that $m=4$ and $n=6$. }
\end{figure}
}
See Figure~\ref{fig:canon_label} for an example of this labelling
for the paths $u''$ and $\ell'$ used in
Figure~\ref{fig:infinite-pair}.

\begin{Def}
Given $m,n \in \mathbb{N}$ and a point $y=(y_1,y_2) \in \ZZ^2$, a
{\it{frame}} is a collection of coloured edges which are anchored
about a point $y$, and have labels as shown in
Figure~\ref{figframe}. We denote this frame by $\Frame{m}{n}{y}$.
\end{Def}
\figureornotfigure{
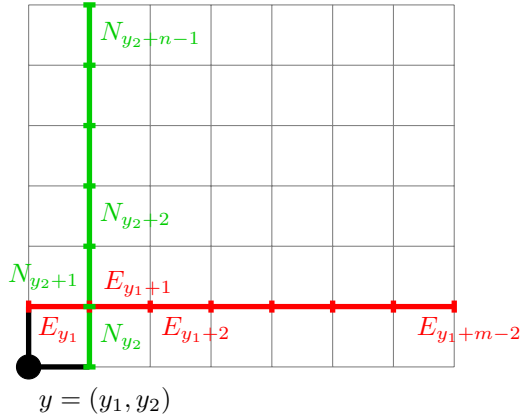
\begin{figure}[!h]
\newcommand{\myzsize}{0.8}
\newcommand{\csty}{\scriptstyle}
\begin{tikzpicture}[scale=\myzsize]
\draw[draw=white!50!black]  (0,0) grid (7,6);
\draw[line width=2,draw=black] (0,1)--(0,0)--(1,0);
\draw[line width=2,draw=dgreen] (1,0)--(1,6);
\draw[line width=2,draw=red] (0,1)--(7,1);
\node (ypoint) at (0,0) {};
\draw[fill=black] (0,0) circle (0.2);
\node [black, below right] at (ypoint.south) {\small $y=(y_1,y_2)$};
\node [red,below] at (0.5,1) {\small $E_{y_1}$};
\node [red,above] at (1.8,1) {\small $E_{y_1+1}$};
\node [red,below] at (2.75,1) {\small $E_{y_1+2}$};
\node [red,below] at (7.25,1) {\small $E_{y_1+m-2}$};
\node [dgreen,right] at (1,0.5) {\small $N_{y_2}$};
\node [dgreen,left] at (1,1.5) {\small $N_{y_2+1}$};
\node [dgreen,right] at (1,2.5) {\small $N_{y_2+2}$};
\node [dgreen,right] at (1,5.5) {\small $N_{y_2+n-1}$};
\foreach \x in {0,1,2,3,4,5,6,7}{
    \draw[line width=2,red] (\x,1-0.1)--(\x,1+0.1);
    }
\foreach \y in {0,1,2,3,4,5,6}{
    \draw[line width=2,dgreen] (1-0.1,\y)--(1+0.1,\y);
    }
\end{tikzpicture}
\caption{The frame $\Frame{m}{n}{y}$.\label{figframe}}
\end{figure}
}

The frame is something we will use to measure distances to steps
in the bi-infinite path from.
\newcommand{\PP}{\mathcal{P}}
\newcommand{\FF}{\mathcal{F}}
\newcommand{\Measure}[2]{#1 \left( #2 \right)}

\begin{Def}
\label{def:frame_one} Let $\FF_y=\Frame{m}{n}{y}$ and $(u'',\ell')
\in B_{m-1,n-1}\times B_{m-1,n}$ where $m,n \in \ZZ$ and
$y=(y_1,y_2) \in \ZZ^2$. Consider the pair of bi-infinite paths
$\PP=((Nu'')^{\ZZ},(E\ell')^{\ZZ})$ (coloured red and green,
respectively) and suppose they are canonically labelled in the
sense outlined above. A {\it{measurement}} of $\PP$ with respect
to a frame $\FF_y$ is a sequence of numbers describing the
horizontal and vertical distances from steps of the frame to steps
of the path which have the same label:
$$\myframeconfiguration{u''}{\ell'}{y} = (a_0,\ldots,a_{n-1},b_0,\ldots,b_{m-2})$$
where $a_i$ is the horizontal distance from step $N_{y_2+i}$ of
$\FF_y$ to the corresponding step in $\PP$ and $b_j$ is the
vertical distance from step $E_{y_1+j}$ of $\FF_y$ to the
corresponding step in $\PP$.
\end{Def}

\begin{Ex}\label{frame_path_example}
Let $m=4$, $n=6$ and $y=(2,7)$. Suppose that the pair of
bi-infinite paths $\PP$ to be the same as in
Figure~\ref{fig:canon_label}. The frame $\Frame{4}{6}{(2,7)}$ is
illustrated in Figure~\ref{fig:frame_measure}. The horizontal
distance from the frames lowest north step $N_7$ to the $N_7$ on
$\PP$ is $2$, so the first entry of
$\myframeconfiguration{u''}{\ell'}{y}$ is $2$. The horizontal
distance from the next lowest north step of the frame $N_8$ to the
corresponding one on the path is also $2$, so the second entry of
$\myframeconfiguration{u''}{\ell'}{y}$ is $2$. For step $N_9$, the
horizontal distance to step $N_9$ on the path $\PP$ is $4$, so the
third entry of $\myframeconfiguration{u''}{\ell'}{y}$ is $4$.
Doing the same for the steps $N_{10}$, $N_{11}$ and $N_{12}$ we
get the values $5$, $5$, and $6$, respectively. The values for
$\myframeconfiguration{u''}{\ell'}{y}$ are $(2,2,4,5,5,6)$ so far.

Next we consider the east steps of the frame from left to right.
The vertical distance from step $E_2$ of the frame to step $E_2$
of $\PP$ is $-4$. This means the next entry of
$\myframeconfiguration{u''}{\ell'}{y}$ is $-4$. For $E_3$, the
vertical distance is $-1$ so the next entry of
$\myframeconfiguration{u''}{\ell'}{y}$ is $-1$. For steps $E_4$
that value is $+1$.

Therefore $\myframeconfiguration{u''}{\ell'}{y} = (2,2,4,5,5,6,-4,-1,1)$.
\figureornotfigure{
\begin{figure}[!h]
\begin{tikzpicture}[scale=0.7, trans/.style={red,thick,->,
>=stealth}]
\draw[draw=white!50!black]  (0,0) grid (10, 15);
\draw[line width=2,draw=red] (4.95, 9.05)--(4.95, 10.05)--(5.95, 10.05)--(5.95, 13.05)--(6.95, 13.05)--(6.95, 15.05)--(7.95, 15.05)--(7.95, 15.05);
\draw[line width=2,draw=dgreen] (5.05, 8.95)--(7.05, 8.95)--(7.05, 9.95)--(8.05, 9.95)--(8.05, 10.95)--(8.05, 10.95)--(8.05, 11.95)--(9.05, 11.95)--(9.05, 12.95)--(9.05, 12.95)--(9.05, 13.95)--(9.05, 13.95)--(9.05, 14.95);
\foreach \x/\ux in {0/1,1/3,2/4,3/7,4/9,5/10,6/13,7/15}{
\draw[red] node at (\x+0.5,\ux+0.4) {\small$E_{\x}$};
}
\foreach \x/\ux in {1/0,1/1,1/2,3/3,4/4,4/5,5/6,5/7,5/8,7/9,8/10,8/11,9/12,9/13,9/14}{
\draw[dgreen] node at (\x+0.5,\ux+0.5) {\small $N_{\ux}$};
}
\draw[line width=2,draw=red] (2.95, 4.05)--(2.95, 7.05)--(3.95, 7.05)--(3.95, 9.05)--(4.95, 9.05)--(4.95, 10.05)--(5.95, 10.05)--(5.95, 10.05);
\draw[line width=2,draw=dgreen] (3.05, 3.95)--(4.05, 3.95)--(4.05, 4.95)--(4.05, 4.95)--(4.05, 5.95)--(5.05, 5.95)--(5.05, 6.95)--(5.05, 6.95)--(5.05, 7.95)--(5.05, 7.95)--(5.05, 8.95)--(7.05, 8.95)--(7.05, 9.95);
\draw[line width=2,draw=red] (1.95, 3.05)--(1.95, 4.05)--(2.95, 4.05)--(2.95, 7.05)--(3.95, 7.05)--(3.95, 9.05)--(4.95, 9.05)--(4.95, 9.05);
\draw[line width=2,draw=dgreen] (2.05, 2.95)--(3.05, 2.95)--(3.05, 3.95)--(4.05, 3.95)--(4.05, 4.95)--(4.05, 4.95)--(4.05, 5.95)--(5.05, 5.95)--(5.05, 6.95)--(5.05, 6.95)--(5.05, 7.95)--(5.05, 7.95)--(5.05, 8.95);
\draw[line width=2,draw=red] (-0.05, 0.05)--(-0.05, 1.05)--(0.95, 1.05)--(0.95, 3.05)--(1.95, 3.05)--(1.95, 4.05)--(2.95, 4.05)--(2.95, 6.05);
\draw[line width=2,draw=dgreen] (0.05, -0.05)--(1.05, -0.05)--(1.05, 0.95)--(1.05, 0.95)--(1.05, 1.95)--(1.05, 1.95)--(1.05, 2.95)--(3.05, 2.95)--(3.05, 3.95)--(4.05, 3.95)--(4.05, 4.95)--(4.05, 4.95)--(4.05, 5.95);
\draw[line width=2,draw=black] (3,7)--(2,7)--(2,8);
\draw[line width=2,draw=red] (2,8)--(5,8);
\draw[line width=2,draw=dgreen] (3,7)--(3,13);
\node (ypoint) at (2,7) {};
\draw[fill=black] (2,7) circle (0.2);
\node [black, rotate=0,below left] at (ypoint) {\small $y=(2,7)$};
\node (jake) [red,below] at (1.5,9.5) {\small $E_2$};
\draw[trans] (jake)--(2.1,8.1);
\node [red,above] at (3.5,8) {\small $E_3$};
\node [red,below] at (4.5,8) {\small $E_4$};
\node [dgreen,left] at (3,7.5) {\small $N_7$};
\node [dgreen,left] at (3,8.5) {\small $N_8$};
\node [dgreen,left] at (3,9.5) {\small $N_9$};
\node [dgreen,left] at (3.1,10.5) {\small $N_{10}$};
\node [dgreen,left] at (3.1,11.5) {\small $N_{11}$};
\node [dgreen,left] at (3.1,12.5) {\small $N_{12}$};
\foreach \x in {2,3,4,5}{
    \draw[line width=2,red] (\x,8-0.1)--(\x,8+0.1);
    }
\foreach \y in {2,3,4,5,6,7,8}{
    \draw[line width=2,dgreen] (2.9,5+\y)--(3.1,5+\y);
    }
\end{tikzpicture}
\caption{The frame and bi-infinite path of Example~\ref{frame_path_example}. Note that $m=4$ and $n=6$.
\label{fig:frame_measure}
}
\end{figure}
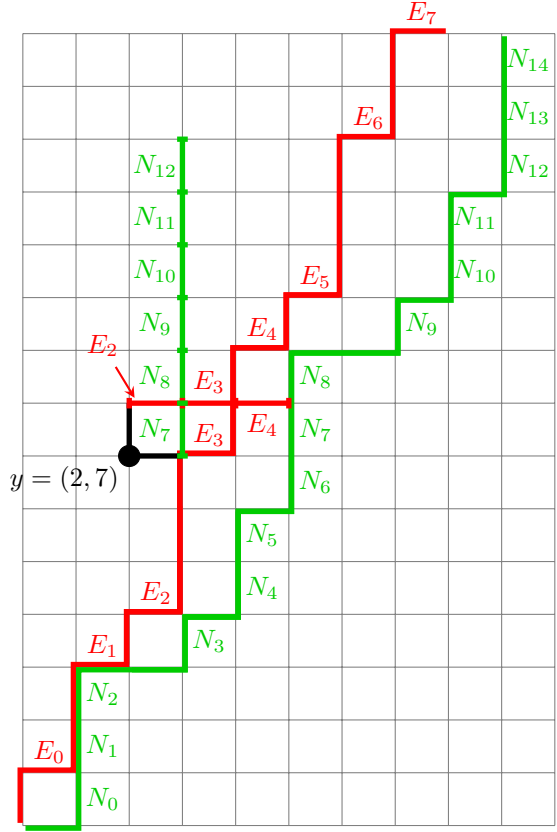
}
\end{Ex}

\subsection{The operator $\varphi$ on $K_{m,n}$}\label{subsec55}

The main result of this subsection is Theorem~\ref{thm55} which
explains the behavior of $\varphi$ on stable configurations in
terms of stable intersections. We will require several technical
lemmas in order to achieve this goal. In order to prove the
required lemmas concerning the frame measurement of paths, we will
need a more algebraic definition of a frame measurement.

\begin{Def}[Equivalent to Definition~\ref{def:frame_one}]
\label{def:frame_two} Let $\FF_y=\Frame{m}{n}{y}$ and $(u'',\ell')
\in B_{m-1,n-1}\times B_{m-1,n}$ where $m,n \in \ZZ$ and
$y=(y_1,y_2) \in \ZZ^2$. Consider the pair of bi-infinite paths
$\PP=((Nu'')^{\ZZ},(E\ell')^{\ZZ})$ and suppose them to be
canonically labelled in the sense outlined above. A
{\it{measurement}} of $\PP$ with respect to a frame $\FF_y$ is a
sequence of numbers describing the horizontal and vertical
distances from steps of the frame to steps of the path which have
the same label:
$$\myframeconfiguration{u''}{\ell'}{y} = (a_0,\ldots,a_{n-1},b_0,\ldots,b_{m-2})$$
where
$$a_i=X_1(N_{y_2+i})-y_1-1 \quad\mbox{ and }\quad b_j=X_2(E_{y_1+j})-y_2-1$$
for all $0\leq i <n$ and $0\leq j <m-1$.
\end{Def}

We will now use this idea of frame measurement to map to
configurations of the sandpile model on $K_{m,n}$ and prove
results concerning them.

\begin{Lem}
\label{lem:frame-implies-compact-range-assumption}
For every $(u'',\ell') \in B_{m-1,n-1}\times B_{m-1,n}$ and $y\in\mathbb{Z}^{2}$, the configuration
$c=(c_1,\ldots,c_{n+m-1})=\myframeconfiguration{u''}{\ell'}{y}$
is a sorted configuration satisfying the compact range assumption.
\end{Lem}

\begin{proof}
  The configuration $c$ is sorted since the sequences $(X_1(N_i))_{i\in\mathbb{Z}}$
  and $(X_2(E_i))_{i\in\mathbb{Z}}$ are weakly increasing in
  the bi-infinite binomial paths $(E\ell')^\mathbb{Z}$ and
  $(Nu'')^\mathbb{Z}$ respectively.

  The configuration $c$ satisfies the compact range assumption on the
  non-sink component $c^{\leq n}$: $N_{y_2}$ (respectively
  $N_{y_2+n-1}$) is the first (respectively last) north step of the
  periodic pattern, which is a conjugate to $E\ell'$ that is a binomial
  path of $B_{m,n}$. Between $N_{y_2}$ and $N_{y_2+n-1}$ there are
  at most $m$ east steps so
$$ m \geq X_1(N_{y_2+n-1})- X_1(N_{y_2})= \left(X_1(N_{y_2+n-1})-y_1-1\right) - \left(X_1(N_{y_2})-y_1-1\right) = c_n-c_1.$$
A similar argument about the east steps $E_{y_1}$ and $E_{y_1+m-2}$ of
$(Nu'')^\mathbb{Z}$ alongside a consideration of the $n$ north steps of a conjugate of
$Nu''$ shows that $c_{m+n-1}-c_{n+1}\leq n$.
Therefore the configuration $c$ satisfies the compact range assumption.
\end{proof}

The following lemma shows that every stable sorted configuration
on $K_{m,n}$ can be described by at least one frame
$\Frame{m}{n}{y}$ and a pair of paths $(u'',\ell')$.

\begin{Lem}
\label{lem:existence-of-frame-description} For any stable sorted
configuration $c=(c_1,\ldots,c_{n+m})$ on $K_{m,n}$ there exists a
triple $(u'',\ell',y)\in B_{m-1,n-1}\times B_{m-1,n}\times
\mathbb{Z}^2$ such that $c =
\myframeconfiguration{u''}{\ell'}{y}.$
\end{Lem}

\begin{proof}
A triple $(u'',\ell',y)$ for which $ c = \myframeconfiguration{u''}{\ell'}{y}$ is given by
\begin{align*}
u''& = (N^{c_{n+1}}E)(N^{c_{n+2}-c_{n+1}}E)\ldots (N^{c_{n+m-1}-c_{n+m-2}}E)N^{n-1-c_{n+m-1}}\\
\ell' & = (E^{c_1}N)(E^{c_2-c_1}N)\ldots (E^{c_{n}-c_{n-1}}N)E^{m-1-c_{n}}\\
y & = (0,0) .
\end{align*}
This triple is well-defined because the configuration is
non-negative, sorted, and stable. We leave it to the reader to
verify that $\myframeconfiguration{u''}{\ell'}{y} = c$.
\end{proof}

The effect of the toppling $T^{\leq n}$ (respectively $T^{>n}$)
used in Algorithm~\ref{alg:phi} may be interpreted as a move of
the frame one unit step to the south (respectively west) without
changing the bi-infinite paths.

\begin{Lem}
\label{lem:toppling-and-frame} For every $(u'',\ell')\in
B_{m-1,n-1}\times B_{m-1,n}$ and $y=(y_1,y_2)\in\mathbb{Z}^2$, we
have
\begin{align*}
T^{\leq n}\left(\myframeconfiguration{u''}{\ell'}{(y_1,y_2)}\right) &= \myframeconfiguration{u''}{\ell'}{(y_1,y_2-1)}\\
T^{>n}\left(\myframeconfiguration{u''}{\ell'}{(y_1,y_2)}\right) &= \myframeconfiguration{u''}{\ell'}{(y_1-1,y_2)}.
\end{align*}
\end{Lem}

\begin{proof}
This proof is illustrated by an example in Figure~\ref{fig:example-frame-move}.
\figureornotfigure{
\begin{figure}
\begin{center}
\begin{tikzpicture}[scale=0.7, trans/.style={red,thick,->,
>=stealth}]
\draw[draw=white!50!black]  (0,0) grid (10, 15);
\draw[line width=2,draw=red] (4.95, 9.05)--(4.95, 10.05)--(5.95, 10.05)--(5.95, 13.05)--(6.95, 13.05)--(6.95, 15.05)--(7.95, 15.05)--(7.95, 15.05);
\draw[line width=2,draw=dgreen] (5.05, 8.95)--(7.05, 8.95)--(7.05, 9.95)--(8.05, 9.95)--(8.05, 10.95)--(8.05, 10.95)--(8.05, 11.95)--(9.05, 11.95)--(9.05, 12.95)--(9.05, 12.95)--(9.05, 13.95)--(9.05, 13.95)--(9.05, 14.95);
\foreach \x/\ux in {0/1,1/3,2/4,3/7,4/9,5/10,6/13,7/15}{
\draw[red] node at (\x+0.5,\ux+0.4) {\small$E_{\x}$};
}
\foreach \x/\ux in {1/0,1/1,1/2,3/3,4/4,4/5,5/6,5/7,5/8,7/9,8/10,8/11,9/12,9/13,9/14}{
\draw[dgreen] node at (\x+0.5,\ux+0.5) {\small $N_{\ux}$};
}
\draw[line width=2,draw=red] (2.95, 4.05)--(2.95, 7.05)--(3.95, 7.05)--(3.95, 9.05)--(4.95, 9.05)--(4.95, 10.05)--(5.95, 10.05)--(5.95, 10.05);
\draw[line width=2,draw=dgreen] (3.05, 3.95)--(4.05, 3.95)--(4.05, 4.95)--(4.05, 4.95)--(4.05, 5.95)--(5.05, 5.95)--(5.05, 6.95)--(5.05, 6.95)--(5.05, 7.95)--(5.05, 7.95)--(5.05, 8.95)--(7.05, 8.95)--(7.05, 9.95);
\draw[line width=2,draw=red] (1.95, 3.05)--(1.95, 4.05)--(2.95, 4.05)--(2.95, 7.05)--(3.95, 7.05)--(3.95, 9.05)--(4.95, 9.05)--(4.95, 9.05);
\draw[line width=2,draw=dgreen] (2.05, 2.95)--(3.05, 2.95)--(3.05, 3.95)--(4.05, 3.95)--(4.05, 4.95)--(4.05, 4.95)--(4.05, 5.95)--(5.05, 5.95)--(5.05, 6.95)--(5.05, 6.95)--(5.05, 7.95)--(5.05, 7.95)--(5.05, 8.95);
\draw[line width=2,draw=red] (-0.05, 0.05)--(-0.05, 1.05)--(0.95, 1.05)--(0.95, 3.05)--(1.95, 3.05)--(1.95, 4.05)--(2.95, 4.05)--(2.95, 6.05);
\draw[line width=2,draw=dgreen] (0.05, -0.05)--(1.05, -0.05)--(1.05, 0.95)--(1.05, 0.95)--(1.05, 1.95)--(1.05, 1.95)--(1.05, 2.95)--(3.05, 2.95)--(3.05, 3.95)--(4.05, 3.95)--(4.05, 4.95)--(4.05, 4.95)--(4.05, 5.95);
\draw[line width=2,draw=black] (3,6)--(2,6)--(2,7);
\draw[line width=2,draw=red] (2,7)--(5,7);
\draw[line width=2,draw=dgreen] (3,6)--(3,12);
\node (ypoint) at (2,6) {};
\draw[fill=black] (2,6) circle (0.2);
\node [black, rotate=0,below left] at (ypoint) {\small $y=(2,6)$};
\node (jake) [red,below] at (1.5,8.5) {\small $E_2$};
\draw[trans] (jake)--(2.1,7.1);
\node [red,below] at (3.5,7) {\small $E_3$};
\node [red,below] at (4.5,7) {\small $E_4$};
\node [dgreen,left] at (3,6.5) {\small $N_6$};
\node [dgreen,left] at (3,7.5) {\small $N_7$};
\node [dgreen,left] at (3,8.5) {\small $N_8$};
\node [dgreen,left] at (3,9.5) {\small $N_{9}$};
\node [dgreen,left] at (3.1,10.5) {\small $N_{10}$};
\node [dgreen,left] at (3.1,11.5) {\small $N_{11}$};
\foreach \x in {2,3,4,5}{
        \draw[line width=2,red] (\x,7-0.1)--(\x,7+0.1);
        }
\foreach \y in {2,3,4,5,6,7,8}{
        \draw[line width=2,dgreen] (2.9,4+\y)--(3.1,4+\y);
        }
\end{tikzpicture}
\caption{\label{fig:example-frame-move}Example of moving the frame of Figure~\ref{fig:frame_measure} one step south. Note that $m=4$ and $n=6$.}
\end{center}
\end{figure}
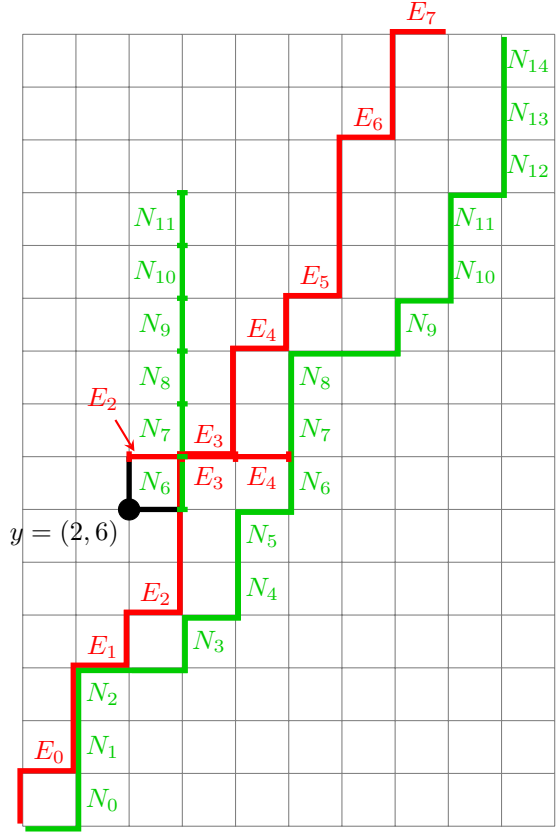
} 
Let
\begin{align*}
\myframeconfiguration{u''}{\ell'}{(y_1,y_2)}&=c = (c_1,\ldots, c_{m+n-1}) \\
\myframeconfiguration{u''}{\ell'}{(y_1,y_2-1)}&=c' = (c'_1,\ldots,c'_{n+m-1}).
\end{align*}
We will describe the configuration $c'$ in terms
of the configuration $c$ by analysing the move of the frames position from
$(y_1,y_2)$ to $(y_1,y_2-1)$.

The decrement of $y_2$ by unity implies that for $i>n$,
$$ c'_i = X_2(E_{y_1+(i-n)-1})-(y_2-1)-1 = (X_2(E_{y_1+(i-n)-1})+1)-y_2-1 = c_i+1.$$
For $i\leq n$, the north steps $(N_{y_2+i-1})_{1\leq i\leq n}$ of
$(E\ell')^\mathbb{Z}$ defining the configuration $c$ (see
Figure~\ref{fig:frame_measure}) become the north steps
$(N_{(y_2-1)+i-1})_{1\leq i \leq n}$ in $c'$ (see
Figure~\ref{fig:example-frame-move}). All steps except the last
step ($N_{y_2+n-1}$) are simply shifted to the next index in this
sequence, and the new first index is $N_{y_2-1}$.

For all of these shifted steps $(N_{y_2+i-1})_{1\leq i< n}$, the
number $X_1(N_{y_2+i-1})-y_1-1$ is unchanged since neither $y_1$
nor $X_1(N_{y_2+i-1})$ changes from $c$ to $c'$.  This implies
that $c'_i = c_{i-1}$ for $2 \leq i \leq n$.

For $i=1$, we remark that the north step $N_{y_2-1}$ that appears
in $(E\ell')^\mathbb{Z}$ which defines $c'_1$ differs from the
disappearing north step $N_{y_2+n-1}$ defining $c_n$ by $n$. This
means that these two north steps are the ``same'' step in the
periodic pattern. This periodic pattern is an $(m,n)$-binomial
path, so $X_1(N_{y_2-1}) = X_1(N_{y_2+n-1})-m$ due to the $m$ east
steps of the periodic pattern. In terms of configurations, it
means $ c'_1 = c_n-m$.

To summarize this discussion we have
$$(c'_1,\ldots,c'_{n+m-1}) = (c_n-m,c_1,\ldots,c_{n-1},c_{n+1}+1,\ldots ,c_{n+m-1}+1) = T^{\leq n}(c)$$
where the rightmost equality comes from
Lemma~\ref{lem:operator-on-compact-range} since, according to
Lemma~\ref{lem:frame-implies-compact-range-assumption}, the
configuration $c$ satisfies the compact range assumption.
The proof for the operator $T^{>n}$ is similar: in particular the
periodic pattern is a $(m-1,n)$-binomial path and we have to consider
$X_2(E_{y_1-1})=X_2(E_{y_1+(m-1)-1})-n$.
\end{proof}

Notice that the operators $T^{\leq n}$ and $T^{>n}$ preserve
toppling and permuting classes since both are the composition of a
toppling and a (cyclic) permutation on one of the components. The
following lemma shows that, with respect to the paths in the
plane, there is only one equivalence class.

\begin{Lem}
  Let $(u'',\ell')\in B_{m-1,n-1}\times B_{m-1,n}$,
  $x=(x_1,x_2)\in \mathbb{Z}^2$ and $y=(y_1,y_2)\in \mathbb{Z}^{2}$.
  The configurations $\myframeconfiguration{u''}{\ell'}{x}$ and
  $\myframeconfiguration{u''}{\ell'}{y}$ are toppling and permuting
  equivalent.
\end{Lem}

\begin{proof}
  Let $z=(z_1,z_2) = \left(\min(x_1,y_1),\min(x_2,y_2)\right)$.  We
  show that both configurations of the lemma are toppling and permuting
  equivalent to the configuration $\myframeconfiguration{u''}{\ell'}{z}$
  via some applications of the operators $T^{\leq n}$ and $T^{>n}$
  which preserve the toppling and permuting classes.

  Indeed, using Lemma~\ref{lem:toppling-and-frame}, we have
$$ \myframeconfiguration{u''}{\ell'}{z} = \left(T^{\leq n}\right)^{x_2-z_2}\cdot \left(T^{>n}\right)^{x_1-z_1}(\myframeconfiguration{u''}{\ell'}{x})$$
and a similar expression exists for $\myframeconfiguration{u''}{\ell'}{y}.$
\end{proof}

To simulate Algorithm~\ref{alg:phi} using frames it remains to
show that the test in the argument of the `while' condition on
line 4 of the algorithm can be realized in this setting.

\begin{Lem}
\label{lem:bounds-on-stable-configurations}
Let $(u'',\ell',y=(y_1,y_2))\in B_{m-1,n-1}\times B_{m-1,n}\times\mathbb{Z}^2$ and let
$c=\myframeconfiguration{u''}{\ell'}{y}$. Then
\begin{enumerate}
\item[(i)] $c^{\leq n} \nlesquare  \delta^{\leq n}\iff$ $X_1(N_{y_2+n-1})> y_1+m$
\item[(ii)] $c^{\leq n} \ngesquare  0^{\leq n}\iff $ $X_1(N_{y_2})\leq y_1$
\item[(iii)] $c^{>n} \nlesquare  \delta^{>n}\iff$ $X_2(E_{y_1+m-2})> y_2+n$
\item[(iv)] $c^{>n} \ngesquare  0^{>n}\iff$ $X_2(E_{y_1})\leq y_2$.
\end{enumerate}
\end{Lem}

\begin{proof}
  Since the configuration $c=\myframeconfiguration{u''}{\ell'}{y}$
  is sorted, the four equivalences are
  respectively equivalent to $c_n>m-1$, $c_1<0$, $c_{n+m-1}>n-1$ and
  $c_{n+1}<0$. The statements in (i)--(iv) give path-wise interpretations of these (simpler) inequalities.
\end{proof}

To complete the description of the algorithm in terms of a moving
frame, it remains to show that the iterates of the operator
$\varphi$ visit all stable intersections. The following lemma
shows that stable intersections of bi-infinite paths are exactly
the bottom-left corner of frames defining stable sorted
configurations.

\begin{Lem}
  Suppose that $(u'',\ell',y)\in B_{m-1,n-1}\times B_{m-1,n}\times \mathbb{Z}^2$.
  Then the configuration $\myframeconfiguration{u''}{\ell'}{y}$ is
  stable if and only if $y$ is a stable intersection of $((Nu'')^\mathbb{Z},(E\ell')^\mathbb{Z})$.
\end{Lem}

\begin{proof}
Let $c = (c_1,\ldots,c_{n+m-1}) = \myframeconfiguration{u''}{\ell'}{y}$ and $y=(y_1,y_2)$.
\begin{itemize}
\item If the configuration $c$ is stable, it means that $0 \lesquare c \lesquare \delta$.
Let us first consider the steps of the path $(E\ell')^\mathbb{Z}$.
Since $0 \leq c_1 \leq c_n \leq m-1$ we deduce from
Lemma~\ref{lem:bounds-on-stable-configurations} that $$X_1(N_{y_2}) >
y_1 \mbox{ and } X_1(N_{y_2+n-1}) \leq y_1+m.$$
Since the periodic
pattern $E\ell'$ contains $m$ east steps and $n$ north steps, we have
$$X_1(N_{y_2-1}) = X_1(N_{y_2+n-1})-m.$$
From these two observations we have
$$X_1(N_{y_2-1}) \leq y_1 < X_1(N_{y_2}).$$
These inequalities imply that $y$ is a vertex of
$(E\ell')^\mathbb{Z}$ and the strict inequality implies that $y$ is
followed by an east step.
A similar discussion for the path $(Nu'')^\mathbb{Z}$ leads to the similar inequalities:
$$ X_2(E_{y_1-1}) \leq y_2 < X_2(E_{y_1}).$$
This shows that $y$ also belongs to $(Nu'')^\mathbb{Z}$ and it is followed by a north step.
Therefore $y$ is a stable intersection in $((Nu'')^\mathbb{Z},(E\ell')^\mathbb{Z})$.
\item
Conversely, we assume that $y$ is a stable intersection.
Since $y$ belongs to $(E\ell')^\mathbb{Z}$, we have the
inequalities
$$X_1(N_{y_2-1}) \leq y_1 < X_1(N_{y_2})$$
where the strict equality comes from the east step following $y$.
Since $X_1(N_{y_2-1+n}) = X_1(N_{y_2-1})+m$ we have
$$X_1(N_{y_2+n-1}) \leq y_1+m.$$
Using Lemma~\ref{lem:bounds-on-stable-configurations} we deduce that
$0^{\leq n}\lesquare c^{\leq n} \lesquare \delta^{\leq n}.$
Since $y$ belongs to $(Nu'')^\mathbb{Z}$ we may deduce, in a similar
manner, that $0^{>n} \lesquare c^{>n} \lesquare \delta^{>n}$.
Thus $c$ is a stable configuration.\qedhere
\end{itemize}
\end{proof}

We can now finally state and prove that the computation of $\varphi$ may be
interpreted in term of a moving frame as a jump from a stable
intersection to the preceding stable intersection, if such an intersection exists.

\begin{Thm}\label{thm55}
Let $(u'',\ell')\in B_{m-1,m-1}\times B_{m-1,n}$ and let $y$ be a
stable intersection of the pair
$\left((Nu'')^\mathbb{Z},(E\ell')^\mathbb{Z}\right)$. Let $x$ be
the next stable intersection of
$((Nu'')^\mathbb{Z},(E\ell')^\mathbb{Z})$ if such an intersection
exists, and $x=y$ otherwise. Then
$$\varphi\left(\myframeconfiguration{u''}{\ell'}{y}\right) =
\myframeconfiguration{u''}{\ell'}{x}.$$
\end{Thm}

\begin{proof}
Let
$\left\{z^{(k)}=\left(z_1^{(k)},z_2^{(k)}\right)\right\}_{0\leq
k<m}$ be the collection of $m$ stable intersections of the pair
$((Nu'')^\mathbb{Z},(E\ell')^\mathbb{Z})$ where $z^{(m-1)}_1<
\ldots < z^{(0)}_1$ and $z^{(m-1)}_2< \ldots < z^{(0)}_2$.
Consider the stable configuration of some stable intersection
$z^{(j)}$: $$c=\myframeconfiguration{u''}{\ell'}{z^{(j)}}.$$ Since
the operator $\varphi$ may be interpreted as a sequence of
applications of $T^{\leq n}$ and $T^{>n}$ corresponding to unit
steps to the south or west, the resulting stable configuration
$c'=\varphi(c)$ is defined by a stable intersection $z^{(i)}$
where $j\leq i$:
 $$c'=\myframeconfiguration{u''}{\ell'}{z^{(i)}}.$$
 If $j=m-1$, then $i=m-1$ and we arrive at the $x=y$ case of the statement.

Otherwise it remains to show that $i=j+1$. We obtain this fact by
a contradiction that involves the minimality of the cardinality of
$A$ in the definition of $\varphi$. Assume that $i>j+1$ and
consider the configuration related to the stable intersection
$z^{(j+1)}$:
$$ c'' = \myframeconfiguration{u''}{\ell'}{z^{(j+1)}}.$$
From the definition of $\varphi(c)$ in our algorithm and Proposition~\ref{prop:two-parameter-description-of-A}, we have
$$A=\{v_{n-k},\ldots,v_n \}\cup\{v_{m+n-1-l},\ldots ,v_{n+m-1}\}$$
where $k=z^{(j)}_2-z^{(i)}_2$ and $l=z^{(j)}_1-z^{(i)}_1$.

However, from Lemma~\ref{lem:toppling-and-frame}, we have
$$c''=(T^{\leq n})^{z^{(j)}_2-z^{(j+1)}_2}\cdot (T^{>n})^{z^{(j)}_1-z^{(j+1)}_2}(c).$$
Since $c''$ is a stable configuration we deduce that the set
$$A'=\{v_{n-k'},\ldots,v_n \}\cup\{v_{m+n-1-l'},\ldots ,v_{n+m-1}\} $$
where $k'=z^{(j)}_2-z^{(j+1)}_2>0$ and
$l'=z^{(j)}_1-z^{(j+1)}_1>0$, is non-empty and thus also a
candidate for the definition of $\varphi(c).$ This means that
$k'<k$ and $l'<l$ so that $A'\neq A$ and $A'\subset A$, and we
arrive at a contradiction to the minimality of $A$ as claimed in
the definition of $\varphi$. Therefore $\varphi(c)=c''$ and we are
done.
\end{proof}

We conclude this subsection by a proposition formalizing the
description of all sorted configurations satisfying the compact range assumption in a
given toppling and permuting equivalence class. We do not need this
general proposition to bring us forward, but deem it worthy of a mention.

\begin{Prop}
  Let $(u'',l')$ be fixed in $B_{m-1,m-1}\times B_{m-1,n}$.  The map
  $P_{u'',l'}:y\mapsto \myframeconfiguration{u''}{l'}{y}$ is a bijection
  between $\mathbb{Z}^2$ and sorted configurations which satisfy the
  compact range assumption and are toppling and permuting equivalent to
  $\myframeconfiguration{u''}{l'}{(0,0)}$.
\end{Prop}

\begin{proof}
To show the surjectivity of $P_{u'',l'}$, we consider $u$ to be a sorted configuration satisfying
 the compact range assumption.  We start by showing that there exists a sorted \emph{stable}
  configuration $v$ such that
  $v=\left(T^{\leq n}\right)^{\alpha}\cdot
  \left(T^{>n}\right)^{\beta}(u)$  for $(\alpha,\beta)\in\mathbb{Z}^2$.
To do this, we first check that $$w=\left[\left(T^{\leq n}\right)^{-n}\cdot
  \left(T^{>n}\right)^{1-m}\right]^{-u_1}\cdot\left\{\left[\left(T^{\leq n}\right)^{-n}\cdot
  \left(T^{>n}\right)^{1-m}\right]^{m}\cdot \left(T^{\leq n}\right)^n\right\}^{\lceil-
\frac{u_{n+1}}{n}\rceil}(u)$$
 is a non-negative configuration since $\left[\left(T^{\leq n}\right)^{-n}\cdot
  \left(T^{>n}\right)^{1-m}\right]^{-u_1}$ corresponds to $-u_1$ topplings of the sink, 
adding exactly $-u_1$ grains to each vertex of $u^{\leq n}$. In the same way, the 
remaining term corresponds to the addition of $n\lceil -\frac{u_{n+1}}{n}\rceil$ grains to each 
vertex of $u^{>n}$.
Then, we may topple in this non-negative configuration $w$ the unstable vertices of maximal 
value in each component, and obtain for $(\gamma,\delta)\in\mathbb{N}^2$, the stable 
configuration 
$$v = \left( T^{\leq n} \right)^{\gamma}\cdot \left( T^{>n} \right)^{\delta} w=\left(T^{\leq n}\right)^{\alpha}\cdot  \left(T^{>n}\right)^{\beta}(u)$$
with $(\alpha,\beta)\in\mathbb{Z}^2$.
  
Now, we apply $\phi^m$ (which is also a combination of operators $T^{\leq n}$ and $T^{>n}$)
 to this $v$. We get the (unique) sorted recurrent configuration $r$ of the toppling class. 
 Hence, any sorted configuration satisfying the compact range assumption is related to the
 sorted recurrent configuration $r$ via the operators $T^{\leq n}$ and $T^{>n}$,
and these operators are invertible when restricted to configurations satisfying the compact range assumption.
Since $r$ is unique, this implies that $u$ and $\myframeconfiguration{u''}{l'}{(0,0)}$ are 
toppling and permuting equivalent and that there exists $(\alpha',\beta')\in \mathbb{Z}^2$ such that:
$$ u = \left(T^{\leq n}\right)^{\alpha'}\cdot
  \left(T^{>n}\right)^{\beta'}(\myframeconfiguration{u''}{l'}{(0,0)})=
\myframeconfiguration{u''}{l'}{y}$$
with $y=(-\beta',-\alpha')$.
This proves the surjectivity of $P_{u'',l'}$.

To prove the injectivity of $P_{u'',l'}$, we define the two following parameters on a configuration $u$:
$$I_1(u)=\sum_{i=1}^{n+m-1} u_i\mbox{ and  }I_2(u)=\sum_{i=1}^{n} u_i.$$
The relations 
$$I_1(T^{\leq n}(u))=I_1(u)-1,\ I_2(T^{\leq n}(u))=I_2(u)-n,\ I_1(T^{> n}(u))=I_1(u)
\mbox{ and } I_2(T^{> n}(u))=I_2(u)+m  $$
show that all $\myframeconfiguration{u''}{l'}{y}$ are distinct when $y$ 
runs over $\mathbb{Z}^2$.
\end{proof}

\subsection{The operator $\psi$ on $K_{m,n}$}

We can give a similar pictorial description for the action of the
operator $\psi$ on sorted stable configurations.

In Proposition \ref{prop:conjugate} the operators $\varphi$ and
$\psi$ were shown to be conjugate. This conjugation used the
involution $\beta$ which sends a configuration $c$ to the
configuration $\delta-c$. To deal with sorted configurations, let
us denote by $\rho$ the element of $S_n\times S_k$ which reverses
the order of entries of both the sink and the
non-sink parts of a configuration.

Now we may write for every sorted configuration $c$ on $K_{m,n}$:
$$ \psi(c) =   \rho\cdot \beta\cdot \varphi\cdot \rho\cdot \beta(c).$$
In this way, we may compute the action of $\psi$ on sorted
configurations through the action of $\varphi$ on the same set.

\begin{Thm}\label{thm61}
Let $(u'',\ell')\in B_{m-1,n-1}\times B_{m-1,n}$ and let $y$ be a
stable intersection of the pair
$\left((Nu'')^\mathbb{Z},(E\ell')^\mathbb{Z}\right)$. Let $x$ be
the preceding stable intersection of
$((Nu'')^\mathbb{Z},(E\ell')^\mathbb{Z})$ if such an intersection
exists, and $x=y$ otherwise. Then
$$\psi\left(\myframeconfiguration{u''}{\ell'}{y}\right) =
\myframeconfiguration{u''}{\ell'}{x}.$$
\end{Thm}

Given a binomial word $u=u_1u_2\ldots u_{k-1}u_{k}$ written as $k$
letters, define the reverse of $u$ to be $\rho(u) =
u_ku_{k-1}\ldots u_2u_1$. The proof of the previous theorem relies
on the following lemma.

\begin{Lem}
For any triple $(u'',\ell',y)\in B_{m-1,n-1}\times B_{m-1,n}
\times  \mathbb{Z}^2$ we have
$$ \rho\cdot\beta\cdot\myframeconfiguration{u''}{\ell'}{(y_1,y_2)} = \myframeconfiguration{\rho(u'')}{\rho(\ell')}{(-y_1,-y_2)}.$$
\end{Lem}

\begin{proof}
If $\myframeconfiguration{u''}{\ell'}{(0,0)} = (a_0,\ldots ,
a_{n-1},b_0,\ldots ,b_{m-2})$ then
\begin{align*}
\ell'&=E^{a_0-0}NE^{a_1-a_0}N\ldots E^{a_{n-1}-a_{n-2}}NE^{m-1-a_{n-1}}\\
\implies \rho(\ell')&=E^{m-1-a_{n-1}}NE^{a_{n-1}-a_{n-2}}N\ldots E^{a_1-a_0}NE^{a_0-0}\\
\noalign{and}
u''&=N^{b_0-0}EN^{b_1-b_0}E\ldots EN^{n-1-b_{m-2}}\\
\implies \rho(u'') &=N^{n-1-b_{m-2}}EN^{b_{m-2}-b_{m-3}}E\ldots
EN^{b_0-0}.
\end{align*}
From the paths for $N\rho(u'')$ and $E\rho(\ell')$ we see that
$$\myframeconfiguration{\rho(u'')}{\rho(\ell')}{(0,0)} = (m-1-a_{n-1},m-1-a_{n-2},\ldots,m-1-a_0,n-1-b_{m-2},\ldots,n-1-b_0).$$
Applying $\beta$ to this sequence gives
$$\beta \cdot \myframeconfiguration{\rho(u'')}{\rho(\ell')}{(0,0)} = (a_{n-1},a_{n-2},\ldots,a_0,b_{m-2},\ldots,b_0).$$
Finally, applying $\rho$ to this configuration gives
\begin{align*}
\lefteqn{\rho \cdot \beta \cdot \myframeconfiguration{\rho(u'')}{\rho(\ell')}{(0,0)}}\\
&= (a_0,\ldots , a_{n-1},b_0,\ldots , b_{m-2})=
\myframeconfiguration{u''}{\ell'}{(0,0)}.
\end{align*}
Hence, the claimed formula is satisfied for $y=(0,0)$

We extend it next to any $y$ as follows. Note that the operators
$T^{\leq n}$ and $T^{> n}$ when restricted to configurations
satisfying the compact range assumption have well-defined inverses
and are mutually commutative. We obtain by inspection, similar to
that given at the start of this proof, the following identities
$$ (T^{\leq n})^{-1} = \rho\cdot\beta\cdot T^{\leq n}\cdot\rho\cdot\beta \quad \mbox{ and }\quad (T^{> n})^{-1} =
\rho\cdot\beta\cdot T^{> n}\cdot\rho\cdot\beta.$$ Then the
following relation, deduced from
Lemma~\ref{lem:toppling-and-frame}, leads to the claim for any
$y\in\mathbb{Z}^2$
$$\myframeconfiguration{u''}{\ell'}{(y_1,y_2)} = (T^{\leq n})^{-y_2}\cdot(T^{> n})^{-y_1}(\myframeconfiguration{u''}{\ell'}{(0,0)}).$$
Indeed, using in addition the fact that $\rho\cdot\beta$ is an
involution, we have
\begin{align*}
\rho\cdot\beta(\myframeconfiguration{u''}{\ell'}{(y_1,y_2)}) & =  \rho\cdot\beta\cdot(T^{\leq n})^{-y_2}\cdot(T^{> n})^{-y_1}(\myframeconfiguration{u''}{\ell'}{(0,0)} \\
 & =   \rho\cdot\beta\cdot(T^{\leq n})^{-y_2}\cdot(T^{> n})^{-y_1}\cdot\rho\cdot\beta(\myframeconfiguration{\rho(u'')}{\rho(\ell')}{(0,0)})\\
 & =  (T^{\leq n})^{y_2}\cdot(T^{> n})^{y_1}(\myframeconfiguration{\rho(u'')}{\rho(\ell')}{(0,0)})\\
 & =  \myframeconfiguration{\rho(u'')}{\rho(\ell')}{(-y_1,-y_2)}.\qedhere
\end{align*}
\end{proof}

\begin{proof}(of Theorem~\ref{thm61})
We have the following equivalences:
$$\begin{array}{ll}
\multicolumn{2}{l}{\mbox{$y=(y_1,y_2)$ is a stable intersection of $((Nu'')^\mathbb{Z},(E\ell')^\mathbb{Z})$}} \\
\iff & \mbox{ $\myframeconfiguration{u''}{\ell'}{(y_1,y_2)}$ is a stable configuration}\\
\iff & \mbox{ $\myframeconfiguration{\rho(u'')}{\rho(\ell')}{(-y_1,-y_2)}$ is a stable configuration}\\
\iff & \mbox{ $(-y_1,-y_2)$ is a stable intersection of
$((N\rho(u''))^\mathbb{Z},(E\rho(\ell'))^\mathbb{Z})$},
\end{array}$$
where in the second equivalence one uses the fact that the
involution $\rho\cdot\beta$ is also an involution when restricted
to stable configurations.

According to Theorem~\ref{thm55}, we consider the $x=(x_1,x_2)$
stable intersection after $(-y_1,-y_2)$ in
$((N\rho(u''))^\mathbb{Z},(E\rho(\ell'))^\mathbb{Z})$, if any, and
$(x_1,x_2)=(-y_1,-y_2)$ otherwise.

We have
$$\begin{array}{lcl}
\psi(\myframeconfiguration{u''}{\ell'}{(y_1,y_2)})&  =  & \rho\cdot\beta\cdot\varphi\cdot\rho\cdot\beta(\myframeconfiguration{u''}{\ell'}{(y_1,y_2)}\\
\ & = & \rho\cdot\beta\cdot\varphi(\myframeconfiguration{\rho(u'')}{\rho(\ell')}{(-y_1,-y_2)})\\
\ & = & \rho\cdot\beta(\myframeconfiguration{\rho(u'')}{\rho(\ell')}{(x_1,x_2)})\\
\ & = & \myframeconfiguration{u''}{\ell'}{(-x_1,-x_2)}).
\end{array}$$
To conclude we observe that $(-x_1,-x_2)$, if different from
$(y_1,y_2)$, is the stable intersection preceding 
$(y_1,y_2)$ in $((Nu'')^\mathbb{Z},(E\ell')^\mathbb{Z})$.
\end{proof}

\subsection{Consequences of the pictorial interpretations}

An interesting consequence of Theorem \ref{thm61} is a pictorial
characterization of the sorted $K_{m,n}$-parking configurations.
\begin{Cor}
  \label{cor:parking} The sorted $K_{m,n}$-parking configurations are
  the stable ones described by a pair of periodic bi-infinite paths 
$((Nu'')^{\mathbb{Z}},(El')^{\mathbb{Z}})$ 
  and a stable intersection $(y_1,y_2)$ such that for $i=1,\ldots ,m-1$,
  the east step $E_{y_1+i}$ of $(El')^{\mathbb{Z}}$ satisfies $\mathrm{pos}(E_{y_1+i})\leq 0$.
\end{Cor}
\begin{proof}
Let $u$ be a sorted stable configuration.
 
If there exists $i$ such that $\mathrm{pos}(E_{y_1+1}) \geq 1$, then there exists a 
stable intersection $(y_1',y_2')$ strictly before $(y_1,y_2)$. 
Let $u'=\myframeconfiguration{u''}{l'}{(y'_1,y'_2)}$. 
Since $\phi^{k}(u')=\left(T^{\leq n}\right)^{\alpha'(k)}\cdot\left(T^{>n}\right)^{\beta'(k)}(u)$ 
for some $(\alpha'(k),\beta'(k))\in\mathbb{N}^2$, we deduce that the fixed point of $\phi$, 
which is the single expected $K_{m,n}$-parking configuration is not before $(y'_1,y'_2)$
and is therefore distinct from $u$.

If for all $i=1,\ldots ,m-1$, $\mathrm{pos}(E_{y_1+1}) \leq 0$ then there is no 
stable intersection strictly before $(y_1,y_2)$ and the description of $\phi$ by 
positive powers of 
$\phi^{k}(u)=\left(T^{\leq n}\right)^{\alpha(k)}\cdot\left(T^{>n}\right)^{\beta(k)}(u)$ 
for some $(\alpha(k),\beta(k))\in\mathbb{N}^2 $ given by the algorithm implies that 
$\phi(u)=u$, hence $u$ is $K_{m,n}$-parking.
\end{proof}

\begin{Rem}\label{rem:inv}
This Remark is a sequel to Remark~\ref{rem:rem}. Theorems
\ref{thm55} and \ref{thm61} show, in particular, that the
operators $\varphi$ and $\psi$ acting on the sorted stable
configurations on $K_{m,n}$ are essentially inverse of each other.

In fact, if $c$ is a sorted stable configuration which is not
recurrent, then $\varphi(\psi(c))=c$, and if $c$ is a sorted
stable configuration which is not parking, then
$\psi(\varphi(c))=c$.

Moreover, in these cases the operators $\psi$ and $\varphi$ are
inverses of each others in the sense of semigroups, i.e.
$\psi(\varphi(\psi(c)))=\psi(c)$ and
$\varphi(\psi(\varphi(c)))=\varphi(c)$ for all sorted stable
configurations $c$.
\end{Rem}

Starting with any sorted stable configuration on $K_{m,n}$, we can
act iteratively with $\psi$ until we get a sorted recurrent
configuration, and then we can move back, acting with $\varphi$,
until we get a sorted $K_{m,n}$-parking configuration. In this
way, we always pass through $m$ distinct sorted stable
configurations, since these configurations correspond to the
stable intersections of the corresponding periodic bi-infinite
paths. Notice also that every sorted stable configuration occurs
in one of these $m$-sets.

This discussion provides the following {\em graduated description}
of all the $m$ sorted stable configurations on $K_{m,n}$ in each
toppling and permuting class.

\begin{Cor}
Let $c$ be the sorted recurrent configuration of a toppling and
permuting class of the sandpile model on $K_{m,n}$. Let
$(Nu''E,E\ell')$ be the parallelogram polyomino describing $c$.
Let $(z^{(0)},\ldots,z^{(m-1)})$ be the ordered stable
intersections of $((Nu'')^\mathbb{Z},(E\ell')^\mathbb{Z})$. Then
the $m$ sorted stable configurations toppling and permuting
equivalent to $c$ are described by
$$\myframeconfiguration{u''}{\ell'}{z^{(k)}} = \varphi^{k}(c)$$
for all $0\leq k \leq m-1$.

Similarly, let $c$ be a sorted $K_{m,n}$-parking configuration of
a toppling and permuting class of the sandpile model on $K_{m,n}$.
Let $(Nu''E,E\ell')$ be the pair of binomial paths describing $c$.
Let $(z^{(0)},\ldots,z^{(m-1)})$ be the ordered stable
intersections of $((Nu'')^\mathbb{Z},(E\ell')^\mathbb{Z})$. Then
the $m$ sorted stable configurations toppling and permuting
equivalent to $c$ are described by
$$\myframeconfiguration{u''}{\ell'}{z^{(k)}} = \psi^{m-k-1}(c)$$
for all $0\leq k \leq m-1$.
\end{Cor}

Figure~\ref{fig:toppling-class} illustrates the $4$ sorted stable
configurations of the toppling and permuting class described by
the example in Figure~\ref{fig:infinite-pair}.\newline

\figureornotfigure{
\begin{figure}[ht!]
\newcommand{\myzsize}{0.60}
\newcommand{\csty}{\footnotesize}
\subfigure[$\csty \varphi^3(u)={0,0,0,2,3,3, \choose
0,2,3,*}$\label{fig:phi3}]{
\begin{tikzpicture}[scale=\myzsize]
\draw[draw=white!50!black]  (0,0) grid (4, 6); \draw[rounded
corners, opacity=0.3, line width=5, draw = orange] (0,0) rectangle
(4, 6); \draw[line width=2,draw=red] (-0.05, 0.05)--(-0.05,
1.05)--(0.95, 1.05)--(0.95, 3.05)--(1.95, 3.05)--(1.95,
4.05)--(2.95, 4.05)--(2.95, 6.05); \draw[line width=2,draw=dgreen]
(0.05, -0.05)--(1.05, -0.05)--(1.05, 0.95)--(1.05, 0.95)--(1.05,
1.95)--(1.05, 1.95)--(1.05, 2.95)--(3.05, 2.95)--(3.05,
3.95)--(4.05, 3.95)--(4.05, 4.95)--(4.05, 4.95)--(4.05, 5.95);
\draw[fill=orange] (0,0) circle (0.2); \draw[orange] node at
(1,-1) {\small$z^{(3)}=(0, 0)$}; \foreach \x/\ux in {2/3,1/2,0/0}{
\draw[red] node at (\x+0.5,-0.5) {\small$\ux$}; } \foreach \y/\uy
in {5/3,4/3,3/2,2/0,1/0,0/0}{ \draw[dgreen] node at (-0.5,\y+0.5)
{\small$\uy$}; } \foreach \x/\rx/\ux in {2/2/4,1/1/3,0/0/1}{
\draw[red] node at (\x+0.6,\ux+0.4) {\small$E_{\rx}$}; } \foreach
\y/\ry/\uy in {5/5/4,4/4/4,3/3/3,2/2/1,1/1/1,0/0/1}{ \draw[dgreen]
node at (\uy+0.6,\y+0.4) {\small$N_{\ry}$}; }
\end{tikzpicture}
} \subfigure[$\csty\varphi^2(u)={0,1,1,2,2,2,\choose
0,3,5,*}$\label{fig:phi2}]{
\begin{tikzpicture}[scale=\myzsize]
\draw[draw=white!50!black]  (0,0) grid (4, 6); \draw[rounded
corners, opacity=0.3, line width=5, draw = orange] (0,0) rectangle
(4, 6); \draw[line width=2,draw=red] (-0.05, 0.05)--(-0.05,
1.05)--(0.95, 1.05)--(0.95, 4.05)--(1.95, 4.05)--(1.95,
6.05)--(2.95, 6.05)--(2.95, 6.05); \draw[line width=2,draw=dgreen]
(0.05, -0.05)--(1.05, -0.05)--(1.05, 0.95)--(2.05, 0.95)--(2.05,
1.95)--(2.05, 1.95)--(2.05, 2.95)--(3.05, 2.95)--(3.05,
3.95)--(3.05, 3.95)--(3.05, 4.95)--(3.05, 4.95)--(3.05, 5.95) --
(4.05,5.95); \draw node at (4.5,0) {}; \draw[fill=orange] (0,0)
circle (0.2); \draw[orange] node at (1,-1) {\small$z^{(2)}=(2, 3)$};
\foreach \x/\ux in {2/5,1/3,0/0}{ \draw[red] node at (\x+0.5,-0.5)
{\small$\ux$}; } \foreach \y/\uy in {5/2,4/2,3/2,2/1,1/1,0/0}{
\draw[dgreen] node at (-0.5,\y+0.5) {\small$\uy$}; } \foreach
\x/\rx/\ux in {2/4/6,1/3/4,0/2/1}{ \draw[red] node at
(\x+0.6,\ux+0.4) {\small$E_{\rx}$}; } \foreach \y/\ry/\uy in
{5/8/3,4/7/3,3/6/3,2/5/2,1/4/2,0/3/1}{ \draw[dgreen] node at
(\uy+0.6,\y+0.4) {\small$N_{\ry}$}; }
\end{tikzpicture}
} \subfigure[$\csty\varphi(u) = {0,0,1,1,1,3, \choose
2,4,5,*}$\label{fig:phi1}]{
\begin{tikzpicture}[scale=\myzsize]
\draw[draw=white!50!black]  (0,0) grid (4, 6); \draw[rounded
corners, opacity=0.3, line width=5, draw = orange] (0,0) rectangle
(4, 6); \draw[line width=2,draw=red] (-0.05, 0.05)--(-0.05,
3.05)--(0.95, 3.05)--(0.95, 5.05)--(1.95, 5.05)--(1.95,
6.05)--(2.95, 6.05)--(2.95, 6.05); \draw[line width=2,draw=dgreen]
(0.05, -0.05)--(1.05, -0.05)--(1.05, 0.95)--(1.05, 0.95)--(1.05,
1.95)--(2.05, 1.95)--(2.05, 2.95)--(2.05, 2.95)--(2.05,
3.95)--(2.05, 3.95)--(2.05, 4.95)--(4.05, 4.95)--(4.05, 5.95);
\draw[fill=orange] (0,0) circle (0.2); \draw[orange] node at
(1,-1) {\small$z^{(1)}=(3, 4)$}; \foreach \x/\ux in {2/5,1/4,0/2}{
\draw[red] node at (\x+0.5,-0.5) {\small$\ux$}; } \foreach \y/\uy
in {5/3,4/1,3/1,2/1,1/0,0/0}{ \draw[dgreen] node at (-0.5,\y+0.5)
{\small$\uy$}; } \foreach \x/\rx/\ux in {2/5/6,1/4/5,0/3/3}{
\draw[red] node at (\x+0.6,\ux+0.4) {\small$E_{\rx}$}; } \foreach
\y/\ry/\uy in {5/9/4,4/8/2,3/7/2,2/6/2,1/5/1,0/4/1}{ \draw[dgreen]
node at (\uy+0.6,\y+0.4) {\small$N_{\ry}$}; }
\end{tikzpicture}
} \subfigure[$\csty u={1,2,2,3,3,3, \choose
0,3,5,*}$\label{fig:phi0}]{
\begin{tikzpicture}[scale=\myzsize]
\draw[draw=white!50!black]  (0,0) grid (4, 6); \draw[rounded
corners, opacity=0.3, line width=5, draw = orange] (0,0) rectangle
(4, 6); \draw[line width=2,draw=red] (-0.05, 0.05)--(-0.05,
1.05)--(0.95, 1.05)--(0.95, 4.05)--(1.95, 4.05)--(1.95,
6.05)--(2.95, 6.05)--(2.95, 6.05); \draw[line width=2,draw=dgreen]
(0.05, -0.05)--(2.05, -0.05)--(2.05, 0.95)--(3.05, 0.95)--(3.05,
1.95)--(3.05, 1.95)--(3.05, 2.95)--(4.05, 2.95)--(4.05,
3.95)--(4.05, 3.95)--(4.05, 4.95)--(4.05, 4.95)--(4.05, 5.95);
\draw[fill=orange] (0,0) circle (0.2); \draw[orange] node at
(1,-1) {\small$z^{(0)}=(5, 9)$}; \foreach \x/\ux in {2/5,1/3,0/0}{
\draw[red] node at (\x+0.5,-0.5) {\small$\ux$}; } \foreach \y/\uy
in {5/3,4/3,3/3,2/2,1/2,0/1}{ \draw[dgreen] node at (-0.5,\y+0.5)
{\small$\uy$}; } \foreach \x/\rx/\ux in {2/7/6,1/6/4,0/5/1}{
\draw[red] node at (\x+0.6,\ux+0.4) {\small$E_{\rx}$}; } \foreach
\y/\ry/\uy in {5/14/4,4/13/4,3/12/4,2/11/3,1/10/3,0/9/2}{
\draw[dgreen] node at (\uy+0.6,\y+0.4) {\small $N_{\ry}$}; }
\end{tikzpicture}
} \caption{The toppling and permuting equivalent sorted stable
configurations from parking to recurrent
configurations.\label{fig:toppling-class}}
\end{figure}
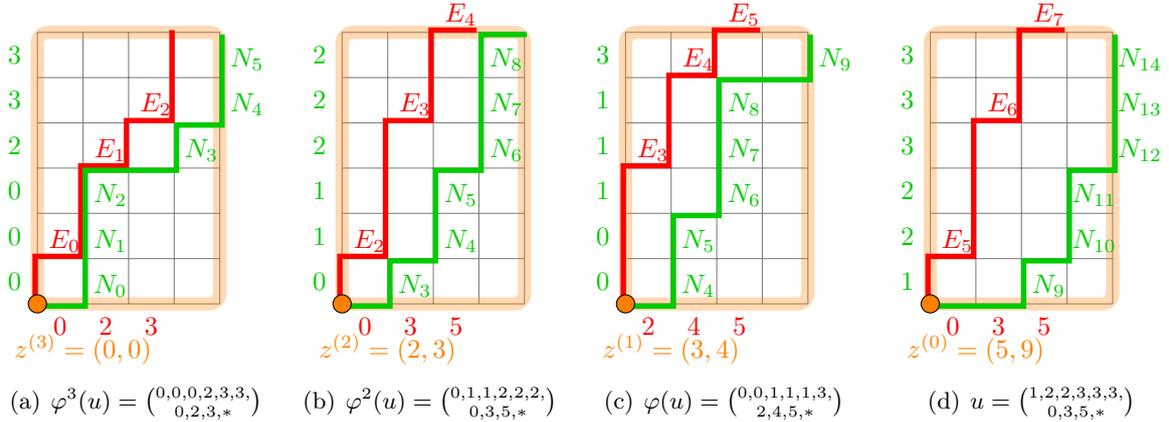
}

As another corollary, we also recover the known bijection (see
\cite{YM}) between sorted recurrent configurations on $K_{m,n}$
and parallelogram polyominoes, since those are exactly the fixed
points of the (pictorial) operator $\psi$.

These descriptions of the extremal stable configurations in the
graduation, be they parking or recurrent, are particular cases of
the following general and \emph{local} description of the grade of
a stable configuration.

For any sorted stable configuration $c$, we have a pair of finite
binomial paths $(u'',\ell')$ which occurs in a pair $p$ of
bi-infinite paths after a stable intersection $y=(y_1,y_2)$. We
consider the $m$ east steps of the (green) factor $E\ell'$, which
are the $(E_{y_1+k})_{0\leq k<m}$ in $p$. We define
$$P^{[y]}_{\geq 1}(c)=\{(y_1+k)\mod m ~:~ 0\leq k < m \mbox{ and } \pos(E_{y_1+k}) \geq 1\}$$
which describes the east steps in $E\ell'$ whose parameter $\pos$
is at least $1$.  The \emph{grade} $\grade(c)$ of a stable
configuration $c$ is defined as the cardinality of $P^{[y]}_{\geq
1}(c)$.

We remark that this definition of grade is not changed by a
translation $t=(t_1,t_2)\in\mathbb{Z}^2$ of the pair $p$ of
bi-infinite path. Indeed, the stable intersection describing $c$
becomes $y+t=(y_1+t_1,y_2+t_2)$ and
$$P^{[y+t]}_{\geq 1}(c) = \{ (x+t_1) \mod m ~:~ x\in P^{[y]}_{\geq 1}(c)\}$$
hence $|P^{[y+t]}_{\geq 1}(c)|=|P^{[y]}_{\geq 1}(c)|$. So an
equivalent and explicitly local definition is
$$\grade(c)=|P^{[(0,0)]}_{\geq 1}(c)|.$$

In Figure~\ref{fig:dual-toppling-class} we reproduce the stable
configurations in Figure~\ref{fig:toppling-class}, mentioning now
the indices of east green steps and north red steps, used to
compute the $\pos(E_{y_1+k})$ in the definition of $P^{[y]}_{\geq
1}$. We draw a circle around the green east steps such that
$\pos(E_{y_i})\geq 1$. By additional convention, the sorted parking
configuration of a toppling and permuting equivalent class is
described by a stable intersection at the origin
$z^{(m-1)}=z^{(3)}=(0,0)$. This additional convention induces a global
choice of stable intersections for the stable configurations of
this class. This convention will be used in the proof of the
following Proposition~\ref{prop:grade}.

\figureornotfigure{
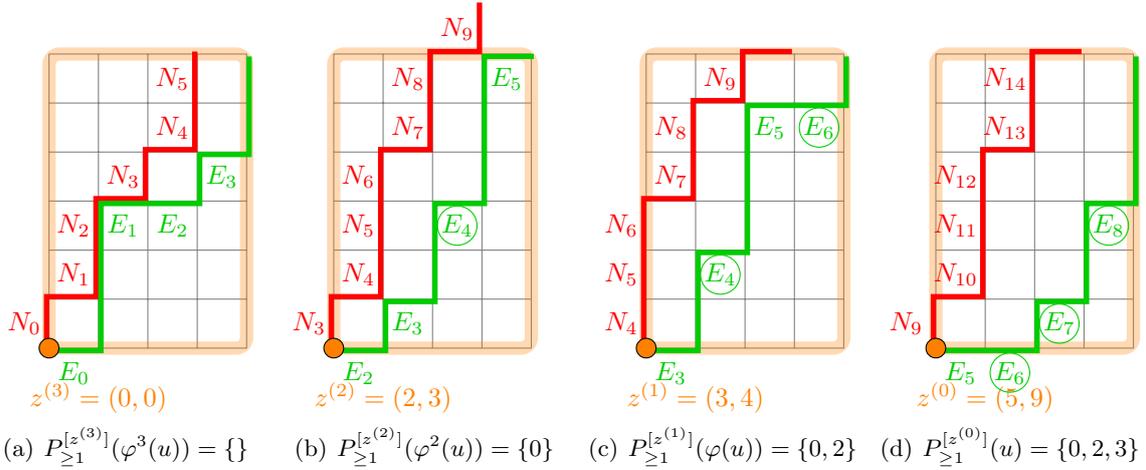
\begin{figure}[ht!]
\newcommand{\myzsize}{0.65}
\newcommand{\csty}{\footnotesize}
\subfigure[$\csty P^{[z^{(3)}]}_{\geq
1}(\varphi^3(u))=\{\}$\label{fig:Pphi3}]{
\begin{tikzpicture}[scale=\myzsize]
\draw[draw=white!50!black]  (0,0) grid (4, 6); \draw[rounded
corners, opacity=0.3, line width=5, draw = orange] (0,0) rectangle
(4, 6); \draw[line width=2,draw=red] (-0.05, 0.05)--(-0.05,
1.05)--(0.95, 1.05)--(0.95, 3.05)--(1.95, 3.05)--(1.95,
4.05)--(2.95, 4.05)--(2.95, 6.05); \draw[line width=2,draw=dgreen]
(0.05, -0.05)--(1.05, -0.05)--(1.05, 0.95)--(1.05, 0.95)--(1.05,
1.95)--(1.05, 1.95)--(1.05, 2.95)--(3.05, 2.95)--(3.05,
3.95)--(4.05, 3.95)--(4.05, 4.95)--(4.05, 4.95)--(4.05, 5.95);
\draw[fill=orange] (0,0) circle (0.2); \draw[orange] node at
(1,-1) {\small$z^{(3)}=(0, 0)$};
\foreach \x/\rx/\ux in {-1/0/0,0/1/1,0/2/2,1/3/3,2/4/4,2/5/5}{
\draw[red] node at (\x+0.5,\ux+0.5) {\small$N_{\rx}$}; } \foreach
\y/\ry/\uy in {-1/0/0,2/1/1,2/2/2,3/3/3}{ \draw[dgreen] node at
(\uy+0.5,\y+0.5) {\small$E_{\ry}$}; } \foreach \y/\x in {}{
\draw[dgreen] (\x+0.5,\y+0.5) circle (0.2); }
\end{tikzpicture}
} \subfigure[$\csty P^{[z^{(2)}]}_{\geq
1}(\varphi^2(u))=\{0\}$\label{fig:Pphi2}]{
\begin{tikzpicture}[scale=\myzsize]
\draw[draw=white!50!black]  (0,0) grid (4, 6); \draw[rounded
corners, opacity=0.3, line width=5, draw = orange] (0,0) rectangle
(4, 6); \draw[line width=2,draw=red] (-0.05, 0.05)--(-0.05,
1.05)--(0.95, 1.05)--(0.95, 4.05)--(1.95, 4.05)--(1.95,
6.05)--(2.95, 6.05)--(2.95, 6.05) -- (2.95,7.05); \draw[line
width=2,draw=dgreen] (0.05, -0.05)--(1.05, -0.05)--(1.05,
0.95)--(2.05, 0.95)--(2.05, 1.95)--(2.05, 1.95)--(2.05,
2.95)--(3.05, 2.95)--(3.05, 3.95)--(3.05, 3.95)--(3.05,
4.95)--(3.05, 4.95)--(3.05, 5.95) -- (4.05,5.95); \draw node at
(4.5,0) {}; \draw[fill=orange] (0,0) circle (0.2); \draw[orange]
node at (1,-1) {\small$z^{(2)}=(2, 3)$};
\foreach \x/\rx/\ux in
{-1/3/0,0/4/1,0/5/2,0/6/3,1/7/4,1/8/5,2/9/6}{ \draw[red] node at
(\x+0.5,\ux+0.5) {\small$N_{\rx}$}; } \foreach \y/\ry/\uy in
{-1/2/0,0/3/1,2/4/2,5/5/3}{ \draw[dgreen] node at (\uy+0.5,\y+0.5)
{\small$E_{\ry}$}; } \foreach \y/\x in {2/2}{ \draw[dgreen]
(\x+0.5,\y+0.5) circle (0.4); }
\end{tikzpicture}
} \subfigure[$\csty P^{[z^{(1)}]}_{\geq 1}(\varphi(u)) =
\{0,2\}$\label{fig:phi9}]{
\begin{tikzpicture}[scale=\myzsize]
\draw[draw=white!50!black]  (0,0) grid (4, 6); \draw[rounded
corners, opacity=0.3, line width=5, draw = orange] (0,0) rectangle
(4, 6); \draw[line width=2,draw=red] (-0.05, 0.05)--(-0.05,
3.05)--(0.95, 3.05)--(0.95, 5.05)--(1.95, 5.05)--(1.95,
6.05)--(2.95, 6.05)--(2.95, 6.05); \draw[line width=2,draw=dgreen]
(0.05, -0.05)--(1.05, -0.05)--(1.05, 0.95)--(1.05, 0.95)--(1.05,
1.95)--(2.05, 1.95)--(2.05, 2.95)--(2.05, 2.95)--(2.05,
3.95)--(2.05, 3.95)--(2.05, 4.95)--(4.05, 4.95)--(4.05, 5.95);
\draw[fill=orange] (0,0) circle (0.2); \draw[orange] node at
(1,-1) {\small$z^{(1)}=(3, 4)$};
\foreach \x/\rx/\ux in {-1/4/0,-1/5/1,-1/6/2,0/7/3,0/8/4,1/9/5}{
\draw[red] node at (\x+0.5,\ux+0.5) {\small$N_{\rx}$}; } \foreach
\y/\ry/\uy in {-1/3/0,1/4/1,4/5/2,4/6/3}{ \draw[dgreen] node at
(\uy+0.5,\y+0.5) {\small$E_{\ry}$}; } \foreach \y/\x in {1/1,4/3}{
\draw[dgreen] (\x+0.5,\y+0.5) circle (0.4); }
\end{tikzpicture}
} \subfigure[$\csty P^{[z^{(0)}]}_{\geq
1}(u)=\{0,2,3\}$\label{fig:phi8}]{
\begin{tikzpicture}[scale=\myzsize]
\draw[draw=white!50!black]  (0,0) grid (4, 6); \draw[rounded
corners, opacity=0.3, line width=5, draw = orange] (0,0) rectangle
(4, 6); \draw[line width=2,draw=red] (-0.05, 0.05)--(-0.05,
1.05)--(0.95, 1.05)--(0.95, 4.05)--(1.95, 4.05)--(1.95,
6.05)--(2.95, 6.05)--(2.95, 6.05); \draw[line width=2,draw=dgreen]
(0.05, -0.05)--(2.05, -0.05)--(2.05, 0.95)--(3.05, 0.95)--(3.05,
1.95)--(3.05, 1.95)--(3.05, 2.95)--(4.05, 2.95)--(4.05,
3.95)--(4.05, 3.95)--(4.05, 4.95)--(4.05, 4.95)--(4.05, 5.95);
\draw[fill=orange] (0,0) circle (0.2); \draw[orange] node at
(1,-1) {\small$z^{(0)}=(5, 9)$};
\foreach \x/\rx/\ux in
{-1/9/0,0/10/1,0/11/2,0/12/3,1/13/4,1/14/5}{ \draw[red] node at
(\x+0.4,\ux+0.5) {\small$N_{\rx}$}; } \foreach \y/\ry/\uy in
{-1/5/0,-1/6/1,0/7/2,2/8/3}{ \draw[dgreen] node at
(\uy+0.5,\y+0.5) {\small $E_{\ry}$}; } \foreach \y/\x in
{-1/1,0/2,2/3}{ \draw[dgreen] (\x+0.5,\y+0.5) circle (0.4); }
\end{tikzpicture}
} \caption{Evaluation of $P^{[z^{(i})]}_{\geq 1}$ from parking to
recurrent
  configurations as in Figure~\ref{fig:toppling-class}.\label{fig:dual-toppling-class}}
\end{figure}
} 

\begin{Prop}
\label{prop:grade} Let $c$ be a sorted stable configuration on
$K_{m,n}$. Let $\parking(c)$ (resp. $\recurrent(c)$) be the
parking (resp. recurrent) configuration in the toppling and
permuting class of $c$. We have
$$ \psi^{\grade(c)}(\parking(c)) = c = \varphi^{m-1-\grade(c)}(\recurrent(c)).$$
\end{Prop}

\begin{proof}
Let $(z^{(i)})_{i=0,\ldots,m-1}$ the $m$ stable intersections in the
pair of paths related to $c$. By convention we assume without loss
of generality that $z^{(m-1)}=(0,0)$. We shall use the notation
$z^{(i)}=(X_1(z^{(i)}),X_2(z^{(i)}))$ for stable intersections as
we did for steps. For each green east step $E_k$, $k\in \mathbb{Z}$, the
east green step $E_{k-m\pos(E_k)}$ belongs to a stable
intersection denoted $z^{(f(k))}$ since
$$\pos(E_{k-m\pos(E_k)}) = \pos(E_k)-\pos(E_k) = 0,$$
using the relation $\pos(E_{k+m})=\pos(E_k)+1$ induced by
periodicities of paths. Hence, we have the equivalence
$$\pos(E_k) \geq 1 \iff X_1(E_k) > X_1(z^{(f(k))}).$$ Using this
equivalence the definition of $P^{[z^{(i)}]}_{\geq
1}(\phi^i(\recurrent(c)))$ becomes
$$ P^{[z^{(i)}]}_{\geq 1}(\varphi^i(\recurrent(c)))=\{ X_1(z^{(j)}) \mod m ~:~ j=m-1,\ldots ,j+1\}$$
since $\{z^{(j)}\}_{j=0, \ldots ,m-1} =
\{z^{(f(E_{X_1(z^{(i)})+k}))}\}_{k=0,\ldots ,m-1}$, $X_1(z^{(0)}) > X_1(z^{(1)}) >
\ldots > X_1(z^{(m-1)})$ and for the step $E_{X_1(z^{(i)})+k}$
$$\begin{array}{lcl} \pos(E_{X_1(z^{(i)})+k}) \geq 1 & \iff & X_1(z^{(f(X_1(z^{(i)})+k))}) <  X_1(z^{(i)})+k \\
\ &  \iff&  X_1(z^{(f(X_1(z^{(i)})+k))}) \leq X_1(z^{(i)})+k-m < X_1(z^{(i)})
\end{array}$$ where the last equivalence uses
$X_1(z^{(f(X_1(z^{(i)})+k))})-(X_1(z^{(i)})+k) \mod m = 0$. This equivalent
definition implies that
$$\grade(\varphi^i(\recurrent(c)))=m-1-i$$
and the proposition follows.
\end{proof}

\section{Some enumerative results}\label{subsec62}
In this section we will present some enumerative results that we
can derive by considering pairs of bi-infinite paths in which one of
the paths has a particularly regular step-like structure. 
Specializations of our Cyclic lemma lead to lattice path enumerations that are new, e.g. Proposition~\ref{prop65}, and already established, e.g. Proposition~\ref{prop63}.

Let $p$ be a binomial word on the alphabet $\{N,E\}$, and which we will call a pattern in this context.
A binomial word $w$
\emph{cyclically matches} the pattern $p$ if $Ew$ may be decomposed as
$Ew=fg$ where $gf=p$.  Let $\mathrm{Cyc}[p]$ be the set of binomial
words that cyclically match $p$. We denote by $\mathrm{Polyo}[p]$ the
polyominoes whose lower path is $p$.

Let $a$, $b$ and $c$ be positive integers. Fix $p =
(E^aN^b)^c$ and consider $\mathrm{Polyo}[(E^aN^b)^c]$, the set
of parallelogram polyominoes having an $ac\times bc$ bounding box and
such that the lower path is $(E^aN^b)^c$.
It transpires that one can restrict the Cyclic Lemma to the pairs in the cartesian product $B_{cb-1,ca-1}\times \mathrm{Cyc}[(E^aN^b)^c]$ in order to
count those parallelogram polyominoes in $\mathrm{Polyo}[(E^aN^b)^c]$.
This gives us the following result which also appears 
in Irving and Rattan~\cite[Cor. 16]{IR} in 2009 and which can be further traced back
to Bonin, de Mier and Noy~\cite[Thm. 8.3]{BMN} in 2003.
The proof of Bonin et al. Theorem 8.3 is a specialization of our more general Cyclic Lemma.

\begin{Prop}\label{prop63}
For all $a,b,c\geq 1$, we have
$$ |\mathrm{Polyo}[(E^aN^b)^c]| =  \frac{1}{c}{ c(b+a)-2 \choose ca-1}.$$
\end{Prop}

\begin{proof}
The set $\mathrm{Polyo}[(E^aN^b)^c]$, as a subset of parallelogram polyominoes having a $ca\times cb$ bounding box,
corresponds to a subset $R^{a,b,c}$ of sorted recurrent configurations on $K_{ac,ab}$.
To be able to apply the Cyclic Lemma, we simply have to identify the set $P^{a,b,c}$ of all possible pairs
in $B_{ca-1,cb-1}\times B_{ca-1,cb}$ deduced from stable intersections of pairs of paths
$((Nu'')^{\mathbb{Z}},(E\ell')^{\mathbb{Z}})$ related to these configurations in $R^{a,b,c}$.

First we provide a necessary condition on $P^{a,b,c}$. Recall that $E\ell'=(E^aN^b)^c$.
At a stable intersection $y$, we have the binomial word $w$ such that
$Ew=\mypathfactor{y}{((E^aN^b)^c)^{\mathbb{Z}}}{c(a+b)}$
is any period of $((E^aN^b)^c)^{\mathbb{Z}}$ which starts with the letter $E$.
So $w$ necessarily cyclically matches $(E^aN^b)^c$, i.e. $w\in \mathrm{Cyc}[(E^aN^b)^c]$.

Next we show that for $(v,w)\in B_{ca-1,cb-1}\times B_{ca-1,cb}$,
the condition $w\in \mathrm{Cyc}[(E^aN^b)^c]$ is also a sufficient condition for $(v,w)\in P^{a,b,c}$.
In this case we remark that the single parallelogram polyomino defined by a stable intersection of
$((Nv)^{\mathbb{Z}},(Ew)^{\mathbb{Z}})$ belongs to $\mathrm{Polyo}[(E^aN^b)^c]$.
This is because its lower path is $Ef$ where $f\in \mathrm{Cyc}[(E^aN^b)^c]$ and $Ef$ ends in the letter $N$,
as it does for any parallelogram polyomino, hence $Ef = (E^aN^b)^c$.
This gives us the following description of pairs of paths involved in these instances of the cyclic lemma:
$$ P^{a,b,c} = B_{ca-1,cb-1} \times \mathrm{Cyc}[(E^aN^b)^c].$$

We remark that $|\mathrm{Cyc}[(E^aN^b)^c]| = a$ since it is easily shown that this set is in bijection with marking one letter $E$ in the factor $E^aN^b$.
Therefore,
\begin{align*}
|\mathrm{Polyo}[(E^aN^b)^c]| &=
\frac{1}{ca}|B_{ca-1,cb-1}||\mathrm{Cyc}[(E^aN^b)^c]|.\qedhere
\end{align*}
\end{proof}

We leave it to the reader to verify the following classical results obtained here as a corollary.

\begin{Cor}\ \\ \vspace*{-1em}
\begin{enumerate}
\item[(i)] $\mathrm{Polyo}[(EN)^{n+1}]$ is in bijection with Dyck
words of semi-length $n$. (In this case the restriction of the
cyclic lemma is essentially the Dvoretsky-Motzkin cyclic lemma.)
\item[(ii)]$\mathrm{Polyo}[(EN^m)^{n}]$ is in bijection with paths
in $B_{n,mn}$ consisting of $n(m+1)$ steps which are above the
line of slope $m$.
\end{enumerate}
\end{Cor}

A symmetry on the parallelogram polyominoes allows us to extend these results and count parallelogram polyominoes whose lower path is $(E^aN^aE^bN^b)^c$.
This kind of periodic conditions seems to be new, in particular it is not covered by Theorem 5 in the work of Chapman, Chow, Khetan, Petrie-Moulton and Waters \cite{CCKMW}. 

\begin{Prop}\label{prop65}
For all $a,b,c\geq 1$ such that $a\neq b$ we have
$$ |\mathrm{Polyo}[(E^aN^aE^bN^b)^c]| = \frac{1}{2c}{2c(a+b)-2 \choose c(a+b) -1}.$$
\end{Prop}

\begin{proof}
The proof is a variation on the previous proof of Proposition~\ref{prop63}.
We reuse the notation $P^{a,b,c}$ to denote similar but now different sets of objects.

In this case $\mathrm{Cyc}[(E^aN^aE^bN^b)^c]$ has cardinality $a+b$ (seen by marking a letter $E$ in the factor $E^aN^aE^bN^b$).
The pairs of paths involved in the restriction of the cyclic lemma are still described by the Cartesian product
$$ P^{a,b,c} = B_{c(a+b)-1,c(a+b)-1}\times \mathrm{Cyc}[(E^aN^aE^bN^b)^c].$$

The main difference is that now the parallelogram polyominoes involved in the cyclic lemma have two possible fixed lower paths.
More precisely, these polyominoes are
$$ \mathrm{Polyo}^{a,b,c} = \mathrm{Polyo}[(E^aN^aE^bN^b)^c] \cup \mathrm{Polyo}[(E^bN^bE^aN^a)^c] $$
which is a disjoint union since $a\neq b$.
Using this in conjunction with the cyclic lemma we get
$$ |\mathrm{Polyo}^{a,b,c}| = \frac{1}{c(a+b)}{2c(a+b)-2 \choose c(a+b) -1}(a+b).$$

Let $\kappa$ by the involutive word morphism defined on letters by $\kappa(E) = N$ and $\kappa(N)=E$.
When $\rho \cdot \kappa$ is applied to upper and lower paths of $\mathrm{Polyo}[(E^aN^aE^bN^b)^c]$
we obtain an involution which maps to $\mathrm{Polyo}[(E^bN^bE^aN^a)^c]$.
Hence
\begin{align*}
|\mathrm{Polyo}[(E^aN^aE^bN^b)^c]| &= |\mathrm{Polyo}[(E^bN^bE^aN^a)^c]|,\\
\noalign{and so}
|\mathrm{Polyo}[(E^aN^aE^bN^b)^c]| &= \frac{1}{2}|\mathrm{Polyo}^{a,b,c}|.  \qedhere
\end{align*}
\end{proof}

\section{Operators $\varphi$ and $\psi$ on $K_n$}\label{subsec63}
In this section, we show how we can derive a description of the operators $\varphi$ and $\psi$
in the case of the complete graph $K_n$ from our results for $K_{m,n}$ by setting $m=n$.
As we shall show in Proposition~\ref{prop:Kn},
the operators $\varphi$ and $\psi$ on $K_n$ may be simulated by (variable) powers
of the operators $\varphi$ and $\psi$ acting on special configurations on $K_{n,n}$.

As in the previous sections, all computations are equivalent up to permutations 
of the entries of the configurations. 
Therefore, and without loss of generality, we will be able to work at the level of orbits, 
ie. to use the {\em sorted} configurations as representatives.

We start with some definitions and notations.
For any vector $v=(v_i)_{i\in I}$, we denote $v\oplus 1=(v_i+1)_{i\in I}$ and $v\ominus 1 = (v_i-1)_{i\in I}$.
A configuration $u$ on $K_{n,n}$ is called \emph{staircase} if
$u^{\leq n}$ is a permutation of $0,1,\ldots ,n-1$. A
configuration $u$ on $K_{n,n}$ is said to be \emph{$0$-free} if
$u_{i}\geq 1$ for all $i=n+1,\ldots ,2n-1$, {\it i.e.} all entries
of $u^{>n}$ are at least equal to $1$.

Now we give a lemma which links topplings in $K_n$ to topplings in $K_{n,n}$.
Roughly speaking, this lemma says that a given toppling in $K_n$ corresponds
to {\em two} topplings in $K_{n,n}$.

\begin{Lem}\label{lem:Kn}
Let $v$ be a configuration on $K_n$, and $u$ be any staircase
configuration on $K_{n,n}$ such that $u^{>n}=v$. Let $j$ be the
(unique) vertex of the non-sink component (i.e. $1\leq j\leq n$)
of $u$ such that $u_j=n-1$. Then for any $i\in\{n+1,\ldots ,
2n-1,2n\}$ (including the sink), we have:
$$ u-\Delta_{i}-\Delta_j = (\eta(u^{\leq n}),v-\Delta_{i-n}) $$
where for any vector $w=(w_i)_{i=1,\ldots ,n}$, we denote
$\eta(w)=(w_i+1\mod n)_{i=1,\ldots ,n}$. In particular,
$u-\Delta_i-\Delta_j$ is also staircase.
\end{Lem}

\begin{proof}
The toppling $\Delta_{i}$ sends $n$ grains to the non-sink
component of $u$. Because $u$ is staircase, this induces exactly
one toppling $\Delta_j$ that sends back $n$ grains to the sink
component (including the vertex $i$). Thus in the sink component,
the configuration is the one obtained by performing $\Delta_{i-n}$
to $v$ in $K_n$: $v_i$ is decreased by $n-1$, the other $v_{i'}$'s
are increased by $1$. For what concerns the non-sink component,
the vertex $j$ loses its $n-1$ grains, the other vertices get $1$
grain each. By observing that $0=(n-1)+1\mod n$, we conclude that
$u^{\leq n}$ is mapped to $\eta(u^{\leq n})$.
\end{proof}

\begin{Prop}\label{prop:Kn}
Let $v$ be a stable configuration on $K_n$, and let $u=(u^{\leq
n},v\oplus 1)$ be a staircase ($0$-free) configuration on
$K_{n,n}$. We have
$$ \varphi(v) = \left(\varphi^k(u)\right)^{>n}\ominus 1 \mbox{ and } \psi(v) = \left(\psi^l(u)\right)^{>n}\ominus 1$$
where $k$ is the minimal positive integer such that $\varphi^k(u)$
is $0$-free and distinct from $u$, if any, otherwise $k=0$, and $l$ is the minimal
positive integer such that $\psi^l(u)$ is $0$-free and distinct from $u$, if any, otherwise $l=0$.
\end{Prop}

\begin{proof}
Let us first examine the assertion on $\varphi$.

We start by observing that the configuration $v$ on $K_n$ is
stable if and only if a staircase configuration of the form
$u=(u^{\leq n},v\oplus 1)$ is stable and $0$-free. Let $v$ be a
stable configuration on $K_n$, and let $u=(u^{\leq n},v\oplus 1)$
be a staircase configuration on $K_{n,n}$.

Suppose first that $\varphi(v)\neq v$, and let $\varphi(v) =
v-\Delta_A$, so that $A$ is the minimal non-empty subset of
$\{1,2,\dots,n\}$ such that $\varphi(v) = v-\Delta_A$ is stable.
Let $B=B^{\leq n}\cup B^{>n}$ where $B^{\leq n} = \{j\mid j\leq n
\mbox{ and } u_j+|A|\geq n\}$ and $B^{>n}= \{n+i\mid i\in A\}$.
First of all notice that, since $u^{\leq n}$ is a permutation of
$\{0,1,2,\dots,n-1\}$, we have $|B^{>n}|=|A|=|B^{\leq n}|$. For
each $i\in B^{>n}$, we can consider $j\leq n$ such that $u_j=n-1$
and apply Lemma \ref{lem:Kn}, getting
$$
u-\Delta_{i}-\Delta_j = \left(\eta(u^{\leq n}),(v\oplus
1)-\Delta_{i-n}\right)= \left(\eta(u^{\leq n}),(v-\Delta_{i-n})\oplus 1\right).
$$
Now we can iterate this application of the lemma with $B\setminus
\{i,j\}=(B^{\leq n}\setminus\{j\})\cup (B^{>n}\setminus \{i\})$,
taking some $k\in B^{>n}\setminus \{i\}$ and $h\in B^{\leq
n}\setminus\{j\}$ such that the $h$-th component of $\eta(u^{\leq
n})$ is equal to $n-1$. At the end of the iteration, we get
$$
u-\Delta_B=\left(\eta^{|A|}(u^{\leq n}),(v-\Delta_A)\oplus 1\right),
$$
so $(u-\Delta_B)^{>n}=(v-\Delta_A)\oplus 1$, or equivalently
$(u-\Delta_B)^{>n}\ominus 1= v-\Delta_A$. In particular
$u-\Delta_B$ is stable and $0$-free.

By the properties of $\varphi$ on $K_{m,n}$ (with $m=n$) that we
proved in this paper, there exists some $k>0$ such that
$u-\Delta_B=\varphi^{k}(u)$. So it remains to show that such a $k$
is minimal with the property that $\varphi^{k}(u)$ is $0$-free.

If this is not the case, let $i$ be such that $0 < i < k$, and
$u'=\varphi^{i}(u)=u-\Delta_C$ is stable and $0$-free, where
$C=C^{\leq n}\cup C^{>n}$ is the partition of $C$ given by the
intersections with the non-sink and the sink components of
$K_{n,n}$. By definition, in the non-sink component of the
staircase configuration $u$ any height \emph{modulo $n$} appears
exactly once. Notice that this property is preserved by any
toppling. Since the resulting configuration $u'$ is stable, $u'$
must be also staircase. So, after the topplings, $u'=u-\Delta_C$
has preserved the number of grains in the non-sink component,
therefore we must have $|C^{\leq n}| = |C^{>n}|$.

As we already observed, by Lemma \ref{lem:Kn}
$$
u'=\left(\eta^{|D|}(u^{\leq n}),(v-\Delta_{D})\oplus 1\right)
$$
where $D=\{c-n~:~ c\in C^{>n}\}$. As $u'$ is $0$-free, $v-\Delta_D$
is stable, but this contradicts the minimality of $A$, since
$D\subsetneq A$. This shows that $k$ is minimal.

If instead $\varphi(v)=v$, then there is no non-empty $A\subseteq
\{1,2,\dots,n\}$ such that $\varphi(v) = v-\Delta_A$ is stable. If
there is a positive integer $k$ such that
$\varphi^k(u)=u-\Delta_C$ is $0$-free, then, by what we observed
earlier, $|C^{\leq n}|=|C^{>n}|$ and
$$
\varphi^k(u)=\left(\eta^{|D|}(u^{\leq n}),(v-\Delta_{D})\oplus 1\right)
$$
where $D=\{c-n~:~ c\in C^{>n}\}$. As $\varphi^k(u)$ is $0$-free,
$v-\Delta_D$ is stable, which implies $D=\emptyset$. Therefore
$C=\emptyset$ and $\varphi^k(u)=u$, which implies $\varphi(u)=u$,
i.e. $u$ is $K_{n,n}$-parking. But this is a contradiction, since
a $K_{n,n}$-parking staircase configuration cannot be $0$-free.

Therefore there is no such positive $k$, hence by definition
$\varphi(v)=v=(\varphi^0(u))^{>n}\ominus 1$ as claimed.

\medskip

The assertion for the operator $\psi$ is proved analogously. First
of all there is an analogue of Lemma \ref{lem:Kn}, that is proved
in the same way: if $v$ is a configuration on $K_n$, and $u$ is
any staircase configuration on $K_{n,n}$ such that $u^{>n}=v$, let
$j$ be the (unique) vertex of the non-sink component (i.e. $1\leq
j\leq n$) of $u$ such that $u_j=0$; then for any $i\in\{n+1,\ldots,
2n-1,2n\}$ (including the sink), we have:
$$ u+\Delta_{i}+\Delta_j = \left(\widetilde{\eta}(u^{\leq n}),v+\Delta_{i-n}\right) $$
where for any vector $w=(w_i)_{i=1,\ldots ,n}$, we denote
$\widetilde{\eta}(w)=(w_i-1\mod n)_{i=1,\ldots ,n}$. In particular,
$u+\Delta_i+\Delta_j$ is also staircase.

Using this, if $\psi(v)\neq v$, then let $\psi(v) = v+\Delta_A$,
so that $A$ is the minimal non-empty subset of $\{1,2,\dots,n\}$
such that $\psi(v) = v+\Delta_A$ is stable. Let $B=B^{\leq n}\cup
B^{>n}$ where $B^{\leq n} = \{j\mid j\leq n \mbox{ and }
u_j-|A|\lneqq 0\}$ and $B^{>n}= \{n+i\mid i\in A\}$. First of all
notice that, since $u^{\leq n}$ is a permutation of
$\{0,1,2,\dots,n-1\}$, we have $|B^{>n}|=|A|=|B^{\leq n}|$. For
each $i\in B^{>n}$, we can consider $j\leq n$ such that $u_j=0$
and apply the previous lemma, getting
$$
u+\Delta_{i}+\Delta_j = \left(\widetilde{\eta}(u^{\leq n}),(v\oplus
1)+\Delta_{i-n}\right)= \left(\widetilde{\eta}(u^{\leq
n}),(v+\Delta_{i-n})\oplus 1\right).
$$
Now, as we did for $\varphi$, we can iterate this application of
the lemma, getting in the end
$$
u+\Delta_B=\left(\widetilde{\eta}^{|A|}(u^{\leq n}),(v+\Delta_A)\oplus
1\right),
$$
so $(u+\Delta_B)^{>n}=(v+\Delta_A)\oplus 1$, or equivalently
$(u+\Delta_B)^{>n}\ominus 1= v+\Delta_A$. In particular
$u+\Delta_B$ is stable and $0$-free.

By the properties of $\psi$ on $K_{m,n}$ (with $m=n$) that we
proved in this paper, there exists some $k>0$ such that
$u+\Delta_B=\psi^{k}(u)$. The proof that such $k$ is minimal with
the property that $\psi^{k}(u)$ is $0$-free is analogous to what
we have done with $\varphi$, and it is omitted.

If instead $\psi(v)=v$, i.e. $v$ is recurrent, then it can be
shown as we did for $\varphi$ that $u$ is also recurrent, i.e.
$\psi(u)=u$. So in this case $k=1$, which is clearly minimal, and
this completes the proof.
\end{proof}
\begin{Ex}
\protect{\rm{ Let $n=5$, and consider the stable configuration
$v=(0,2,2,3,*)$ on $K_n$. Then in this case we can take
$u=(0,1,2,3,4;1,3,3,4,*)$. Now $\varphi(u)=(1,2,3,4,0;2,4,4,0,*)$,
which is not $0$-free, but $\varphi^2(u)=(3,4,0,1,2;4,1,1,2,*)$.
And indeed $\varphi(v)=(3,0,0,1,*)=(\varphi^2(u))^{>5}\ominus 1$
as predicted. }}
\end{Ex}

\section*{Acknowledgments}
The authors thank Robert Cori for bringing Chottin's paper to their attention 
and Einar Steingr\'imsson for helpful comments on the presentation.

\end{document}